\documentclass{amsart}

\usepackage{graphicx}
\usepackage{amsmath}
\usepackage{amsthm}
\usepackage{amsfonts}
\usepackage{amssymb, amscd}

\usepackage{color}

\usepackage[all]{xy}

\newtheorem{thm}{Theorem}[section]

\newtheorem{lem}[thm]{Lemma}
\newtheorem{proposition}[thm]{Proposition}
\newtheorem{example}[thm]{Example}
\theoremstyle{definition}
\newtheorem{definition}[thm]{Definition}
\theoremstyle{remark}
\newtheorem{remark}[thm]{Remark}
\numberwithin{equation}{section}

\newcommand{\K}{\mathbb K}

\newcommand{\dl}{\displaystyle}

\begin{document}

\title[   Cohomology of Hom-Leibniz and $n$-ary Hom-Nambu-Lie superalgebras]
{Cohomology of Hom-Leibniz and $n$-ary Hom-Nambu-Lie superalgebras}%
\author{K.Abdaoui,  S. Mabrouk and  A. Makhlouf}%
\address{K.Abdaoui , Universit\'{e} de Sfax,  Facult\'{e} des Sciences, Sfax Tunisia}%
\email{abdaouielkadri@hotmail.com}
\address{Sami  Mabrouk, Universit\'{e} de Gafsa,  Facult\'{e} des Sciences, Gafsa Tunisia}%
\email{ Mabrouksami00@yahoo.fr}

\address{Abdenacer Makhlouf, Universit\'{e} de Haute Alsace,  Laboratoire de Math\'{e}matiques, Informatique et Applications,
4, rue des Fr\`{e}res Lumi\`{e}re F-68093 Mulhouse, France}%
\email{Abdenacer.Makhlouf@uha.fr}

\thanks {
}

%
\begin{abstract}
The aim of this paper is to study the cohomology of Hom-Leibniz superalgebras. We construct the $q$-deformed Heisenberg-Virasoro superalgebra of Hom-type and  provide as application the computations
of the derivations and second cohomology group. Moreover, we  extend to graded case the Takhtajan's construction
of a cohomology of $n$-ary Hom-Nambu-Lie algebras starting from  cohomology of Hom-Leibniz algebras.
\end{abstract}
\maketitle
\begin{center}
\emph{To the memory of
Faouzi AMMAR}
\end{center}


\section*{Introduction}
In \cite{loday}, J.-L. Loday introduced a non-skewsymmetric version of Lie algebras,
whose bracket satisfies the Leibniz identity. They are called Leibniz algebras.
The Leibniz identity, combined with antisymmetry, is a variation of the Jacobi
identity, hence Lie algebras are skewsymmetric Leibniz algebras.
Recently  Leibniz superalgebras were studied
in \cite{AlbeverioAyupovOmirov,Dzhumadil'daev} and \cite{LiuHu}, etc..\\

Hartwig, Larsson and Silvestrov introduced, in \cite{HartwigLarssonSilvestrov}, Hom-Lie algebras  as a part
of a study of deformations of  Witt and  Virasoro algebras,  while the graded case was
considered in \cite{LarssonSilvestrov}. A Hom-Leibniz superalgebra is a triple
$(\mathcal{A}, [\cdot, \cdot],\alpha)$, in which $\mathcal{A}$ is a $\mathbb{Z}_2$-graded vector space and $\alpha$ is an  even endomorphism of $\mathcal{A}$  satisfying an $\alpha$-twisted variant of the Jacobi identity
$$[[x,y],\alpha(z)]=[\alpha(x),[y,z]]-(-1)^{|x||y|}[\alpha(y),[x,z]],\ \ \forall\ x,\ y,\ z\in \mathcal{H}(\mathcal{A}).$$
Generalizations of Leibniz algebras and Lie algebras, called $n$-ary Nambu algebras,  appeared first in statistical mechanics \cite{Baxter1, Baxter2}. Moreover, Nambu mechanics \cite{Nam}
involves an $n$-ary product $[\cdot,...,\cdot]$ that satisfies the $n$-ary Nambu identity, which is an $n$-ary generalization
of the Jacobi identity.
\begin{eqnarray*}
  && \big[x_1,....,x_{n-1},[y_1,....,y_{n}]\big]= \\
&& \sum_{i=1}^{n}\big[y_1,....,y_{i-1},[x_1,....,x_{n-1},y_i]
  ,y_{i+1},...,y_n\big],\ \forall\ (x_1,..., x_{n-1})\in \mathcal{N}^{ n-1},\ (y_1,...,  y_n)\in \mathcal{N}^{ n}.
  \end{eqnarray*}
These generalizations include $n$-ary Hom-algebra structures generalizing the $n$-ary
algebras of Lie type such as $n$-ary Nambu algebras, $n$-ary Nambu-Lie algebras and $n$-ary Lie algebras. See also \cite{makh, yau1, yau2, yau3} .\\

The purpose of this paper is to study Hom-Leibniz superalgebras and $n$-ary Hom-Nambu-Lie superalgebras. Section 1 is dedicated to basics, we recall  definition of the Hom-Leibniz superalgebras, introduce  the representations and the derivations of  Hom-Leibniz superalgebras and prove as application a construction of  $q$-deformed Heisenberg-Virasoro superalgebra. In the second Section we provide a cohomology of Hom-Leibniz superalgebras and compute the derivations and scalar
second cohomology group of $q$-deformed Heisenberg-Virasoro superalgebra. In the last Section,  we show a relationship  between Hom-Leibniz superalgebras and $n$-ary Hom-Nambu-Lie superalgebras. Moreover, we define a cohomology of $n$-ary Hom-Nambu-Lie superalgebras and  we generalize to twisted situation and graded case, the process used by Daletskii and Takhtajan \cite{Takhtajan0} to relate cohomologies of $n$-ary Hom-Nambu-Lie superalgebras and  Hom-Leibniz superalgebras.
\section{Hom-superdiagebras and Hom-Leibniz superalgebras} Throughout this paper, we will for simplicity of exposition assume that $\mathbb{K}$ is an algebraically closed
field of characteristic zero, even though for most of the general definitions and results in the paper this
assumption is not essential.

 A vector space $V$ is said to be a $\mathbb{Z}_2$-graded if we are given a family $(V_i)_{i\in\mathbb{Z}_2}$ of vector subspace of $V$ such that $V=V_0\oplus V_1.$  The symbol $|x|$ always implies that $x$ is a
$\mathbb{Z}_2$-homogeneous element and $|x|$ is the $\mathbb{Z}_2$-degree. In the sequel, we will denote by $\mathcal{H(A)}$ the set of all homogeneous elements of $\mathcal{A}$ and $\mathcal{H}( \mathcal{A}^n)$ refers to the set of tuples with homogeneous elements.
\subsection{Definitions}
\begin{definition}
A Hom-Leibniz superalgebra is a $\mathbb{Z}_2$-graded vector space $\mathcal{A}=\mathcal{A}_0\oplus \mathcal{A}_1$ over a field $\K$ equipped with a bilinear map $[\cdot,\cdot]:\mathcal{A}\times \mathcal{A}\rightarrow \mathcal{A}$, such that $[\mathcal{A}_i,\mathcal{A}_j]\subset \mathcal{A}_{i+j}$,$\forall\ i,j\in \mathbb{Z}_2$ and an even linear map $\alpha:\mathcal{A}\rightarrow \mathcal{A}$ satisfying
\begin{equation}\label{super-LeiIdent}
[[x,y],\alpha(z)]=[\alpha(x),[y,z]]-(-1)^{|x||y|}[\alpha(y),[x,z]],\ \ \forall\ x,\ y,\ z\in \mathcal{H}(\mathcal{A}).
\end{equation}
The identity \eqref{super-LeiIdent} is called \emph{Super-Hom-Leibniz} identity.
\end{definition}
Suppose that $(\mathcal{A},[\cdot,\cdot],\alpha)$ is a Hom-Leibniz superalgebra. For any $x \in \mathcal{H}(\mathcal{A})$, we define $ad_x \in End_\mathbb{K}(\mathcal{A})$ by $ad_x(y)=[x,y]$, for any $y \in \mathcal{H}(\mathcal{A})$. Then the Super-Hom-Leibniz identity \eqref{super-LeiIdent} can be written as
\begin{eqnarray}\label{ad-hom}
ad_{\alpha(z)}([x, y])&=& [ad_z(x), \alpha(y)] + (-1)^{|z||x|}[\alpha(x),ad_z(y)]
\end{eqnarray}
for all $x, y \in\mathcal{H}(\mathcal{A})$.
\begin{definition}\begin{enumerate} \item A Hom-Leibniz superalgebra $(\mathcal{A},[.,.],\alpha)$ is  called Hom-Lie superalgebra if the bracket $[.,.]$ is skew-symmetric that is
$$[x,y]=-(-1)^{|x||y|}[y,x],\ \ \forall\ x,\ y\in \mathcal{H}(\mathcal{A}).$$
\item A Hom-Leibniz superalgebra $(\mathcal{A},[.,.],\alpha)$ is  called \textbf{multiplicative} Hom-Leibniz superalgebra if
$$\alpha([x,y])=[\alpha(x),\alpha(y)],\ \ \forall\ x,\ y\in \mathcal{H}(\mathcal{A}).$$
\end{enumerate}
\end{definition}
\begin{definition}
Let $f:(\mathcal{A},[\cdot,\cdot],\alpha)\longrightarrow(\mathcal{A}',[\cdot,\cdot]',\alpha')$ be a map. The map $f$ is called
\begin{enumerate}
  \item  even $($resp. odd $)$ map if $f(\mathcal{A}_i)\subset \mathcal{A}_i'$ $($resp. $f(\mathcal{A}_i)\subset \mathcal{A}_{i+1}'$ $)$, for $i=0,1$.
  \item a \textbf{weak morphism} of Hom-Leibniz
superalgebras  if
  \begin{eqnarray*}f ([x,y])&=&
[f (x),f (y)]'.
\end{eqnarray*}
  \item A \textbf{ morphism} of Hom-Leibniz
superalgebras  is a weak morphism of Hom-Leibniz
superalgebras  such that
  \begin{eqnarray*}
f \circ \alpha&=&\alpha'\circ f.
\end{eqnarray*}
  \item An \textbf{automorphism} of Hom-Leibniz
superalgebras  is a morphism of Hom-Leibniz
superalgebras  which is  bijective.
\end{enumerate}
\end{definition}
\subsection{Representation of Hom-Leibniz superalgebras}
Let $(\mathcal{A},[\cdot,\cdot],\alpha)$ be a Hom-Leibniz superalgebra.
\begin{definition}
A linear map $D:\mathcal{A}\rightarrow \mathcal{A}$ is called  $\alpha^k$-derivation of Hom-Leibniz superalgebra for $k\geq0$ if
\begin{equation}
D([x,y])=[D(x),\alpha^k(y)]+(-1)^{|x||D|}[\alpha^k(x),D(y)],\ \ \textrm{for\ all} \ x,y\in \mathcal{H}(\mathcal{A}).
\end{equation}
\end{definition}
\begin{definition}We call a $\mathbb{Z}_2$-graded space $V=V_0\oplus V_1$ a module over $\mathcal{A}$ if there are linear map $\beta:M\rightarrow M$ and two bilinear maps
$$[\cdot,\cdot]:\mathcal{A}\times V\longrightarrow V \ \textrm{and} \ [\cdot,\cdot]:V\times \mathcal{A}\longrightarrow V $$
satisfying the following five axioms, for all $x,y\in \mathcal{H}(\mathcal{A}),~~~v\in \mathcal{H}(V)$
\begin{align}
&\label{eg1-rep}\beta([v,x])=[\beta(v),\alpha(x)],\\
&\label{eg2-rep}\beta([x,v])=[\alpha(x),\beta(v)],\\
&\label{eg3-rep}[[x,y],\beta(v)]=[\alpha(x),[y,v]]-(-1)^{|x||y|}[\alpha(y),[x,v]],\\
&\label{eg4-rep}[[x,v],\alpha(y)]=[\alpha(x),[v,y]]-(-1)^{|x||v|}[\beta(v),[x,y]],\\
&\label{eg5-rep}[[v,x],\alpha(y)]=[\beta(v),[x,y]]-(-1)^{|v||x|}[\alpha(x),[v,y]].
\end{align}
\end{definition}
\begin{example}
Let $(\mathcal{A},[\cdot,\cdot],\alpha)$ be a multiplicative Hom-Leibniz superalgebra. $\mathcal{A}$ is a module over $\mathcal{A}$, where the operator  is the twist map $\alpha$.
\end{example}
\subsection{$q$-deformed Heisenberg-Virasoro superalgebra of Hom-type} $\ $\\
 $\ \ $
In the following, we describe $q$-deformed Heisenberg-Virasoro superalgebra of Hom-type and compute its
derivations and the second scalar cohomology group.\\
Let $\mathcal{A}$ be the complex superalgebra $\mathcal{A} = \mathcal{A}_0\oplus \mathcal{A}_1$ where $\mathcal{A}_0 = \mathbb{C}[t, t^{
-1}]$ is the Laurent polynomials
in one variable and $\mathcal{A}_1 = \theta \mathbb{C}[t, t^{
-1}]$, where $\theta$ is the Grassman variable $(\theta^2 = 0)$. We assume that $t$
and $\theta$ commute. The generators of $\mathcal{A}$ are of the form $t^n$ and $\theta t^n$ for $n \in \mathbb{Z}$.\\
Let $q\in \mathbb{C}\backslash\{0, 1\}$ and $n\in  \mathbb{N}$, we set $\{n\} = \frac{1-q^n}{1-q}$, a $q$-number. The $q$-numbers have the following
properties $\{n + 1\} = 1 + q\{n\} = \{n\} + q^n$ and $\{n + m\} = \{n\} + q^n\{m\}$.

Let $\mathfrak{A}_q$  be a superspace with basis $\{L_m,\ I_m| m\in\mathbb{Z}\}$ of parity $0$ and $\{ G_m,\ T_m| m\in\mathbb{Z}\}$ of parity $1$, where $L_m=-t^mD,\ I_m=-t^m,\ G_m=-\theta t^mD,\ T_m=-\theta t^m$ and $D$ is a $q$-derivation on $\mathcal{A}$ such that
$$D(t^m)=\{m\}t^m,\ D(\theta t^m)=\{m+1\}\theta t^m.$$
We define the bracket  $[\cdot,\cdot]_q:\mathfrak{A}_q\times\mathfrak{A}_q\longrightarrow\mathfrak{A}_q$, with respect the super-skew-symmetry for $n,m\in\mathbb{Z}$  by
\begin{align}
&\label{crochet1}[L_m,L_n]_q=(\{m\}-\{n\})L_{m+n},\\
&\label{crochet2}[L_m,I_n]_q=-\{n\}I_{m+n},\\
&\label{crochet3}[L_m,G_n]_q=(\{m\}-\{n+1\})G_{m+n},\\
&\label{crochet4}[I_m,G_n]_q=\{m\}T_{m+n},\\
&\label{crochet5}[L_m,T_n]_q=-\{n+1\}T_{m+n},\\
&\label{crochet6}[I_m,I_n]_q=[I_m,T_n]_q=[T_m,G_n]_q=[T_m,T_n]_q=[G_m,G_n]_q=0.
\end{align}

Let $\alpha$ be an even linear map on $\mathfrak{A}_q$  defined on the generators by
\begin{eqnarray*}
\alpha_q(L_n)&=&(1+q^n)L_n,\hskip0.5cm \alpha_q(I_n)=(1+q^n)I_n,\\
\alpha_q(T_n)&=&(1+q^{n+1})G_n,~~\alpha_q(T_n)=(1+q^{n+1})T_n.
\end{eqnarray*}
\begin{proposition}
The triple $(\mathfrak{A}_q, [\cdot,\cdot]_q, \alpha_q)$ is a Hom-Lie superalgebra, called the
\textbf{\emph{$q$-deformed Heisenberg-Virasoro superalgebra of Hom-type}}.
\end{proposition}
\begin{proof}We just check
\begin{equation}
\circlearrowleft_{X,Y,Z}(-1)^{|X||Z|}[\alpha_q(X),[Y,Z]_q]_q=0,\ \forall \ X,\ Y,\ Z\in \mathcal{H}(\mathfrak{A}_q).
\end{equation}
Let $L_m,\ L_r,\ I_n,\ G_k,\ T_l$ be a homogeneous elements of $\mathfrak{A}_q$, then we have
\begin{eqnarray*}
&&\circlearrowleft_{L_m,I_n,G_k}(-1)^{|L_m||G_k|}[\alpha_q(L_m),[I_n,G_k]_q]_q\\
&=& \circlearrowleft_{L_m,I_n,G_k}[\alpha_q(L_m),[I_n,G_k]_q]_q\\
&=& [\alpha_q(L_m),[I_n,G_k]_q]_q+[\alpha_q(I_n),[G_k,L_m]_q]_q+[\alpha_q(G_k),[L_m,I_n]_q]_q\\
&=& (1+q^m)[L_m,[I_n,G_k]_q]_q+(1+q^n)[I_n,[G_k,L_m]_q]_q+(1+q^{k+1})[G_k,[L_m,I_n]_q]_q\\
&=& \Big(-(1+q^m)\{n\}\{n+k+1\}+(1+q^n)(\{k+1\}-\{m\})\{n\}+(1+q^{k+1})\{n\}\{m+n\}\Big)T_{m+n+k}\\
&=& 0,
\end{eqnarray*}
and
\begin{eqnarray*}
&& \circlearrowleft_{L_m,L_r,T_l}(-1)^{|L_m||T_l|}[\alpha_q(L_m),[L_r,T_l]_q]_q\\
&=& \circlearrowleft_{L_m,L_r,T_l}[\alpha_q(L_m),[L_r,T_l]_q]_q\\
&=& [\alpha_q(L_m),[L_r,T_l]_q]_q+[\alpha_q(L_r),[T_l,L_m]_q]_q+[\alpha_q(T_l),[L_m,L_r]_q]_q\\
&=& (1+q^m)[L_m,[L_r,T_l]_q]_q+(1+q^r)[L_r,[T_l,L_m]_q]_q+(1+q^{l+1})[T_l,[L_m,L_r]_q]_q\\
&=& \Big((1+q^m)\{l+1\}\{r+l+1\}-(1+q^r)\{l+1\}\{m+l+1\}-(1+q^{l+1})(\{m\}-\{r\})\{l\}\Big)T_{m+n+l}\\
&=& 0.
\end{eqnarray*}
By the same calculation we can proof the other equality.
This ends the proof.
\end{proof}
\section{Cohomology of Hom-Leibniz superalgebras}
In the following we define a cohomology of Hom-Leibniz superalgebras and    describe the cohomology space $H^2_0(\mathfrak{A}_q,\mathbb{C})$ of $q$-deformed Heisenberg-Virasoro superalgebra $\mathfrak{A}_q$ with trivial representation.
 \subsection{Definitions} Let $(\mathcal{A},[\cdot,\cdot],\alpha)$ be a Hom-Leibniz superalgebra and $(V,[\cdot,\cdot]_V,\beta)$ be a representation of $(\mathcal{A},[\cdot,\cdot],\alpha)$.

  \begin{definition}
  A $n$-hom-cochain on $\mathcal{A}$ with values in $V$ is defined to be an $n$-cochain $\varphi \in C^{n}(\mathcal{A},V)$ such that it is
compatible with $\alpha$ and $\beta$ in the sense that $\beta\circ \varphi = \varphi\circ \alpha^{\otimes n}$, i.e.
\begin{eqnarray}\label{alpha-beta-fi}
\beta \circ \varphi(x_{1},...,x_{n}) = \varphi(\alpha(x_{1}),...,\alpha(x_{n})).
\end{eqnarray}\end{definition}
We denote by $C^{n}_{\alpha,\beta}(\mathcal{A},V)$ the set of $n$-hom-cochains:
$ C^{n}_{\alpha,\beta}(\mathcal{A},M)= \{\varphi \in C^{n}(\mathcal{A},V) : \beta \circ \varphi = \varphi \circ \alpha^{\otimes n}\}.$
Let $x_1,... , x_k$ be $k$ homogeneous elements of $\mathcal{A}$. We denote by $|(x_1, ... , x_k)| = |x_1|+ ...+|x_k|~~(\mod 2)$ the parity of an element $(x_1, . . . , x_k)$ in $\mathcal{H}(\mathcal{A}^k)$.\\
A $n$-hom-cochain  $\varphi$ is called even (resp. odd) when we have $\varphi(x_1, . . . , x_k)\in V_{0}$ (resp. $\varphi(x_1, . . . , x_k)\in V_{1})$
for all even (resp odd ) element $(x_1 , . . . , x_k )\in \mathcal{H}(\mathcal{A}^k)$.
\begin{definition}
We define a map $\delta^{n}:C^{n}(\mathcal{A},V)\longrightarrow C^{n+1}(\mathcal{A},V)$ by setting
\begin{eqnarray}
& &\label{Leibnizcohomo}\delta^{n}(f)(x_{1},...,x_{n+1})\\
&=&\Big[\alpha^{n-1}(x_{1}),\varphi(x_{2},...,x_{n+1})\Big]\nonumber\\
    &+& \sum\limits_{i=2}^{n+1}(-1)^{i+|x_{i}|(|\varphi|+|x_{i+1}|+...+|x_{n+1}|)}
    \Big[\varphi(x_{1},...,\widehat{x_{i}},...,x_{n+1}),\alpha^{n-1}(x_{i})\Big]\nonumber\\
    &+&\sum\limits_{1\leq i < j \leq n+1}(-1)^{j+1+|x_{j}|(|x_{i+1}|+...+|x_{j-1}|)}\varphi\Big(\alpha(x_{1}),...,\alpha(x_{i-1}),[x_i,x_j],
    \alpha(x_{i+1}),...,\widehat{x_{j}},...,\alpha(x_{n+1})\Big)\nonumber
 \end{eqnarray}
 \end{definition}
In the sequel we assume that the Hom-Leibniz superalgebra $(\mathcal{A},[\cdot,\cdot],\alpha)$ is multiplicative.
\begin{lem}With the above notations, for any $\varphi \in C^{n}_{\alpha,\beta}(\mathcal{A},V)$, we have
$$\delta^{n}(\varphi)\circ \alpha = \beta \circ \delta^{n}(\varphi) .$$

\end{lem}
\begin{proof}Let $\varphi \in C^{n}_{\alpha,\beta}(\mathcal{A},V)$ and $(x_{1},...,x_{n+1})\in \mathcal{H}(\mathcal{A}^{n+1})$.
\begin{eqnarray*}
   && \delta^{n}(\varphi)\circ \alpha \Big(x_{1},...,x_{n+1}\Big)\\
   &=& \delta^{n}(f)\Big(\alpha(x_{1}),...,\alpha(x_{n+1})\Big)\\&=&\Big[\alpha^{n}(x_{1}),\varphi\Big(\alpha(x_{2}),...,\alpha(x_{n+1})\Big)\Big]
    + \sum\limits_{i=2}^{n+1}(-1)^{i+|x_{i}|(|\varphi|+|x_{i+1}|+...+|x_{n+1}|)}
    \Big[\varphi(\alpha(x_{1}),...,\widehat{x_{i}},...,\alpha(x_{n+1})),\alpha^{n}(x_{i})\Big]\\
    &+&\sum\limits_{1\leq i < j \leq n+1}(-1)^{j+1+|x_{j}|(|x_{i+1}|+...+|x_{j-1}|)}\varphi\Big(\alpha^{2}(x_{1}),...,\alpha^{2}(x_{i-1}),[\alpha(x_i),\alpha(x_j)],
    \alpha^{2}(x_{i+1}),...,\widehat{x_{j}},...,\alpha^{2}(x_{n+1})\Big)\\
    &=&\Big[\alpha^{n}(x_{1}),\varphi \circ \alpha(x_{2},...,x_{n+1})\Big]
    + \sum\limits_{i=2}^{n+1}(-1)^{i+|x_{i}|(|\varphi|+|x_{i+1}|+...+|x_{n+1}|)}
    \Big[\varphi\circ \alpha(x_{1},...,\widehat{x_{i}},...,x_{n+1}),\alpha^{n}(x_{i})\Big]\\
    &+&\sum\limits_{1\leq i < j \leq n+1}(-1)^{j+1+|x_{j}|(|x_{i+1}|+...+|x_{j-1}|)}\varphi\circ \alpha\Big(\alpha(x_{1}),...,\alpha(x_{i-1}),[x_i,x_j],
    \alpha(x_{i+1}),...,\widehat{x_{j}},...,\alpha(x_{n+1})\Big)\\
    &=&\beta \circ\Big[\alpha^{n-1}(x_{1}),\varphi(x_{2},...,x_{n+1})\Big]
    + \sum\limits_{i=2}^{n+1}(-1)^{i+|x_{i}|(|\varphi|+|x_{i+1}|+...+|x_{n+1}|)}
   \beta\circ \Big[\varphi(x_{1},...,\widehat{x_{i}},...,x_{n+1}),\alpha^{n-1}(x_{i})\Big]\\
    &+&\sum\limits_{1\leq i < j \leq n+1}(-1)^{j+1+|x_{j}|(|x_{i+1}|+...+|x_{j-1}|)}\beta\circ \varphi\Big(\alpha(x_{1}),...,\alpha(x_{i-1}),[x_i,x_j],
    \alpha(x_{i+1}),...,\widehat{x_{j}},...,\alpha(x_{n+1})\Big)\\
    &=&\beta\circ \delta^{n}(\varphi)\Big(x_{1},...,x_{n+1}\Big)
\end{eqnarray*}
which completes the proof.
\end{proof}
\begin{thm} Let $(\mathcal{A},[.,.],\alpha)$ be a Hom-Leibniz superalgebra and $(V,\beta)$ be an $\mathcal{A}$-module.
Then the pair $(\bigoplus_{n\geq 0}C_{\alpha,\beta}^{n}, \delta^{n})$ is a cohomology complex, that is  the maps $\delta^{n}$ satisfy  $\delta^{n+1}\circ \delta^{n}=0,~~\forall~~ n \geq 1.$
\end{thm}
\begin{proof}Let $n=1$. Since
\begin{eqnarray*}
\delta^{1}(\varphi)(x_{1},x_{2})&=& [x_{1},\varphi(x_{2})]+(-1)^{|x_{2}||\varphi|}[\varphi(x_{1}),x_{2}]-\varphi([x_{1},x_{2}]).
\end{eqnarray*}
and
\begin{eqnarray*}
&& \delta^{2}(\varphi)(x_{1},x_{2},x_{3})\\
&=&[\alpha(x_{1}),\varphi(x_{2},x_{3})]+(-1)^{|x_{2}|(|\varphi|+|x_{3}|)}[\varphi(x_{1},x_{3}),\alpha(x_{2})]
-(-1)^{|x_{3}||\varphi|}[\varphi(x_{1},x_{2}),\alpha(x_{3})]\\
&-&\varphi([x_{1},x_{2}],\alpha(x_{3}))+(-1)^{|x_{2}||x_{3}|}\varphi([x_{1},x_{3}],\alpha(x_{2}))
+\varphi(\alpha(x_{1}),[x_{2},x_{3}]),
\end{eqnarray*}
then
\begin{eqnarray*}
&& \delta^{2}\circ \delta^{1} (\varphi)(x_{1},x_{2},x_{3})\\
&=&[\alpha(x_{1}),\delta^{1}(\varphi)(x_{2},x_{3})]+(-1)^{|x_{2}|(|\varphi|+|x_{3}|)}[\delta^{1}(\varphi)(x_{1},x_{3}),\alpha(x_{2})]
-(-1)^{|x_{3}||\varphi|}[\delta^{1}(\varphi)(x_{1},x_{2}),\alpha(x_{3})]\\
&-& \delta^{1}(\varphi)([x_{1},x_{2}],\alpha(x_{3}))+(-1)^{|x_{2}||x_{3}|}\delta^{1}(\varphi)([x_{1},x_{3}],\alpha(x_{2}))
+\delta^{1}(\varphi)(\alpha(x_{1}),[x_{2},x_{3}])\\
&=& [\alpha(x_{1}),[x_{2},\varphi(x_{3})]]+(-1)^{|x_{3}||\varphi|}[\alpha(x_{1}),[\varphi(x_{2}),x_{3}]]-[\alpha(x_{1}),\varphi([x_{2},x_{3}])]
+(-1)^{|x_{2}|(|\varphi|+|x_{3}|)}[[x_{1},\varphi(x_{3})],\alpha(x_{2})]\\
&+&(-1)^{|x_{2}|(|\varphi|+|x_{3}|)+|\varphi||x_{3}|}[[\varphi(x_{1}),x_{3}],\alpha(x_{2})]
-(-1)^{|x_{2}|(|\varphi|+|x_{3}|)}[\varphi([x_{1},x_{3}]),\alpha(x_{2})]-(-1)^{|x_{3}||\varphi|}[[(x_{1},\varphi(x_{2})],\alpha(x_{3})]\\
&-&(-1)^{|\varphi|(|x_{2}|+|x_{3}|)}[[\varphi(x_{1}),x_{2}],\alpha(x_{3})]
+ (-1)^{|x_{3}||\varphi|}[\varphi([x_{1},x_{2}]),\alpha(x_{3})]-[[x_{1},x_{2}],\varphi(\alpha(x_{3}))]\\
&-&(-1)^{|x_{3}||\varphi|}[\varphi([x_{1},x_{2}]),\alpha(x_{3})]
+ \varphi([[x_{1},x_{2}],\alpha(x_{3})])+(-1)^{|x_{2}||x_{3}|}[[x_{1},x_{3}],\varphi(\alpha(x_{2}))]\\
&+&(-1)^{|x_{2}||x_{3}|+|x_{2}||\varphi|}[\varphi([x_{1},x_{3}]),\alpha(x_{2})]
-(-1)^{|x_{2}||x_{3}|}\varphi([[x_{1},x_{3}],\alpha(x_{2})])+[\alpha(x_{1}),\varphi([x_{2},x_{3}])]\\
&+&(-1)^{|\varphi|(|x_{2}|+|x_{3}|)}[\varphi(\alpha(x_{1})),[x_{2},x_{3}]]
-\varphi([\alpha(x_{1}),[x_{2},x_{3}]]).
\end{eqnarray*}
Since $\varphi\circ \alpha=\beta \circ \varphi$ and
\begin{eqnarray*}
&&\bullet\quad \varphi([[x_{1},x_{2}],\alpha(x_{3})])-\varphi([\alpha(x_{1}),[x_{2},x_{3}]])+(-1)^{|x_{2}||x_{3}|}\varphi([[x_{1},x_{3}],\alpha(x_{2})])\\
&&= \varphi([[x_{1},x_{2}],\alpha(x_{3})])-[\alpha(x_{1}),[x_{2},x_{3}]])+(-1)^{|x_{2}||x_{3}|}[[x_{1},x_{3}],\alpha(x_{2})])= 0,\\
&&\bullet\quad(-1)^{|\varphi|(|x_{2}|+|x_{3}|)}[\beta(\varphi(x_{1})),[x_{2},x_{3}]]+(-1)^{|x_{2}|(|\varphi|+|x_{3}|)+|\varphi||x_{3}|}[[\varphi(x_{1}),x_{3}],\alpha(x_{2})]\\
&&-(-1)^{|\varphi|(|x_{2}|+|x_{3}|)}[[\varphi(x_{1}),x_{2}],\alpha(x_{3})] =0,\\
&&\bullet\quad(-1)^{|x_{2}||x_{3}|}[[x_{1},x_{3}],\beta(\varphi(x_{2}))]+(-1)^{|x_{3}||\varphi|}[\alpha(x_{1}),[\varphi(x_{2}),x_{3}]]
-(-1)^{|x_{3}||\varphi|}[[(x_{1},\varphi(x_{2})],\alpha(x_{3})]=0,\\
&&\bullet\quad [\alpha(x_{1}),[x_{2},\varphi(x_{3})]]+(-1)^{|x_{2}|(|\varphi|+|x_{3}|)}[[x_{1},\varphi(x_{3})],\alpha(x_{2})]-[[x_{1},x_{2}],\beta(\varphi(x_{3}))]=0,
\end{eqnarray*}
we have $\delta^{2}\circ \delta^{1} (\varphi)(x_{1},x_{2},x_{3})=0$. Our conclusion holds.\\
In general, to prove $\delta^{n+1}\circ \delta^{n}=0$ for $n\geq 2$, we proceed by induction. Suppose
$$\delta^{n+1}\circ \delta^{n}(x_{1},...,x_{n+2})=X_1+X_2+X_3+X_4+X_5+X_6+X_7,$$
where
\begin{eqnarray*}
  \star)\ X_1&=& [\alpha^{n}(x_{1}),[\alpha^{n-1}(x_{2}),\varphi(x_{3},...,x_{n+2})]] \\
   &+&\sum \limits _{3\leq i\leq n+2}(-1)^{i+1+|x_i|(|\varphi|+|x_{i+1}|+...+|x_{n+2}|)} [\alpha^{n}(x_{1}),[\varphi(x_{2},...,\widehat{x_i},...,x_{n+2}),\alpha^{n-1}(x_{i})]]\\
   &+&\sum \limits _{2\leq i\leq n+2}(-1)^{i+|x_j|(|x_{i+1}|+...+|x_{j-1}|)} [\alpha^{n-1}(x_{i}),[\varphi(x_{1},...,\widehat{x_i},...,x_{n+2}),\alpha^{n}(x_{i})]]\\
   &+&\sum \limits _{3\leq i <j\leq n+2} (-1)^{i+j+|x_i|(|\varphi|+|x_{i+1}|+...+|\widehat{x_j}|+...+|x_{n+2}|)}
   [[\varphi(x_{1},...,\widehat{x_i},...,\widehat{x_j},...,x_{n+2}),\alpha^{n-1}(x_{i})],\alpha^{n}(x_{j})]\\
   &-&\sum \limits _{3\leq i <j\leq n+2}(-1)^{i+j+|x_j|(|\varphi|+|x_{j+1}|+...+|x_{n+2}|)} [[\varphi(x_{1},...,\widehat{x_i},...,\widehat{x_j},...,x_{n+2}),\alpha^{n-1}(x_{j})],\alpha^{n}(x_{i})],
\\& & \\ \star)\
   X_2&=& \sum \limits _{2\leq i<j\leq n+2}(-1)^{|x_j|(|x_{i+1}|+...+|x_{j-1}|)} [\alpha^{n}(x_{1}),\varphi(\alpha(x_{2}),...,[x_i,x_j],...,\widehat{x_j},...,\alpha(x_{n+2})),\alpha^{n-1}(x_{i})]\\
   &+&\sum \limits _{2\leq i<j<k\leq n+2} (-1)^{|x_k|(|x_{k+1}|+...+|x_{n+2}|)} [\varphi(\alpha(x_{1}),...,[x_i,x_j],...,\widehat{x_j},...,\widehat{x_k},...,\alpha(x_{n+2})),\alpha^{n}(x_{k})]\\
   &+&\sum \limits _{2\leq i<j<k\leq n+2} (-1)^{|x_k|(|\varphi|+|x_{k+1}|+...+|x_{n+2}|)} [\varphi(\alpha(x_{1}),...,\widehat{x_i},...,[x_j,x_k],...,\widehat{x_k},...,\alpha(x_{n+2})),\alpha^{n}(x_{i})]\\
   &+&\sum \limits _{2\leq i<j<k\leq n+2}(-1)^{|x_j|(|\varphi|+|x_{j+1}|+...+|\widehat{x_k}|+...+|x_{n+2}|)}[\varphi(\alpha(x_{1}),...,[x_i,x_k],...
   ,\widehat{x_j},...,\widehat{x_k},...,\alpha(x_{n+2})),\alpha^{n}(x_{j})],
\\& & \\ \star)\
   X_3&=& \sum \limits _{2\leq i\leq n+2}(-1)^{i+1+|x_i|(|\varphi|+|x_{2}|+...+|x_{i-1}|)} [[\alpha^{n-1}(x_{1}),\alpha^{n-1}(x_{i})],\varphi(\alpha(x_{2}),...,\widehat{x_i},...,\widehat{x_j},...,\alpha(x_{n+2}))]\\
   &+&\sum \limits _{2\leq i<j\leq n+2}(-1)^{i+j+1+(|x_i|+|x_j|)(|\varphi|+|x_{i+1}|+...+|\widehat{x_j}|+...+|x_{n+2}|)} \\ &&\quad\quad[\varphi(\alpha(x_{1}),...,\widehat{x_i}...,\widehat{x_j},...,\alpha(x_{n+2})),[\alpha^{n-1}(x_{i}),\alpha^{n-1}(x_{j})]],
\end{eqnarray*}
\begin{eqnarray*}
  \star)\ X_4&=& \sum \limits _{1\leq i<j<k\leq n+2}(-1)^{j+k+1+(|x_j|+|x_{k}|)(|x_{k-1}|+...+|\widehat{x_j}|+...+|x_{i+1}|)} \\ &&\varphi(\alpha^{2}(x_{2}),...,[[x_i,x_j],\alpha(x_k)],...,\widehat{x_j},...,\widehat{x_k},...,\alpha^{2}(x_{n+2}))]\\
   &-&\sum \limits _{1\leq i<j<k\leq n+2}(-1)^{j+k+1+(|x_j|+|x_{k}|)(|x_{k-1}|+...+|\widehat{x_j}|+...+|x_{i+1}|)} \\ &&\varphi(\alpha^{2}(x_{1}),...,[\alpha(x_i),[x_j,x_k]],...,\widehat{x_j}...,\widehat{x_k},...,\alpha^{2}(x_{n+2}))\\
   &-&\sum \limits _{1\leq i<j<k\leq n+2}(-1)^{j+k+1+(|x_j|+|x_{k}|)(|x_{k-1}|+...+|\widehat{x_j}|+...+|x_{i+1}|)} \\ &&\varphi(\alpha^{2}(x_{1}),...,[[x_i,x_k],\alpha(x_j)]],...,\widehat{x_j}...,\widehat{x_k},...,\alpha^{2}(x_{n+2})),
\\& & \\ \star)\
X_5&=& \sum \limits _{1\leq i<j<k\leq n+2}(-1)^{|x_{k}|(|x_{j+1}|+...+|x_{k-1}|)+|x_{i}|(|x_{2}|+...+|x_{i-1}|)} \\
&& \varphi([\alpha(x_{1}),\alpha(x_i)],\alpha^{2}(x_{2}),...,\widehat{\alpha(x_i)},...,[\alpha(x_j),\alpha(x_k)],...,\widehat{\alpha(x_k)}
,...,\alpha^{2}(x_{n+2}))\\
&+&\sum \limits _{1\leq i<j<k\leq n+2} (-1)^{|x_{k}|(|x_{k-1}|+...+|\widehat{x_j}|+...+|x_{i+1}|)+|x_{j}|(|x_{2}|+...+|[x_i,x_k]|+...+|x_{j-1}|)} \\ &&\varphi([\alpha(x_{1}),\alpha(x_j)],\alpha^{2}(x_{2}),...,[\alpha(x_i),\alpha(x_k)],...,\widehat{\alpha(x_j)}...,\widehat{\alpha(x_k)},...,\alpha^{2}(x_{n+2}))\\
&+&\sum \limits _{1\leq i<j<k\leq n+2} (-1)^{|x_{j}|(|x_{i+1}|+...+|x_{j-1}|)+|x_{k}|(|x_{2}|+...+|[x_i,x_j]|+...+|x_{k-1}|)} \\ &&\varphi([\alpha(x_{1}),\alpha(x_k)],\alpha^{2}(x_{2}),...,[\alpha(x_i),\alpha(x_j)],...,\widehat{\alpha(x_j)}...,\widehat{\alpha(x_k)},...,\alpha^{2}(x_{n+2}))\\ &+&\sum \limits _{1\leq i<j<k<l\leq n+2}(-1)^{|x_{l}|(|x_{j+1}|+...+|x_{l-1}|)} \varphi(\alpha^{2}(x_{1}),...,[x_i,x_k],...,[x_j,x_l],...,\widehat{x_l}...,\alpha^{2}(x_{n+2}))\\
&+&\sum \limits _{1\leq i<j<k<l\leq n+2}(-1)^{|x_{l}|(|x_{l-1}|+...+|\widehat{x_k}|+...+|x_{j+1}|)+|x_{k}|(|x_{k-1}|+...+|[x_j,x_l]|+...+|x_{i+1}|)}\\ &&\varphi(\alpha^{2}(x_{1}),...,[x_i,x_k],...,[x_j,x_l],...,\widehat{x_k},...,\widehat{x_l}...,\alpha^{2}(x_{n+2}))\\
&+&\sum \limits _{1\leq i<j<k<l\leq n+2}(-1)^{|x_{l}|(|x_{l-1}|+...+|\widehat{x_k}|+...+|x_{i+1}|)+|x_{k}|(|x_{k-1}|+...+|x_{j+1}|)}\\
&&\varphi(\alpha^{2}(x_{1}),...,[x_i,x_l],...,[x_j,x_k],...,\widehat{x_k},...,\widehat{x_l}...,\alpha^{2}(x_{n+2})).
\end{eqnarray*}
In addition, $X_6=-X_2,~~X_7=-X_5$. By (\ref{eg3-rep}) and (\ref{alpha-beta-fi}), we can easily check that $X_1+X_3=0$ and $X_4=0$. Thus
$$ \delta^{n+1}\circ \delta^{n}(x_{1},...,x_{n+2})=0,$$
that is, $\delta^{n+1}\circ \delta^{n}=0$ for $n\geq 2$. This completes the proof of this Theorem.
\end{proof}

\begin{definition}
Let $(\mathcal{A},[.,.],\alpha)$ be a Hom-Leibniz superalgebra and $(V,\beta)$ be an $\mathcal{A}$-module.
Then we define
\begin{enumerate}
  \item The $p$-cocycle space
  $$Z_{\alpha,\beta}^{p}(\mathcal{A},V)=Ker\ d^p=\{\varphi \in C_{\alpha,\beta}^{p}(\mathcal{A},V)/d^p\varphi=0\}$$
  \item The $p$-cobord space
  $$B_{\alpha,\beta}^{p}(\mathcal{A},V)=Im\ d^{p-1}=\{\varphi \in C_{\alpha,\beta}^{p}(\mathcal{A},V)/\exists \psi \in C_{\alpha,\beta}^{p-1}(\mathcal{A},V),\ \varphi=d^{p-1}\psi\}$$
\end{enumerate}
\end{definition}
\begin{remark} \
\begin{enumerate}
  \item The $p$-cocycle $Z_{\alpha,\beta}^{p}(\mathcal{A},V)$ space is $\mathbb{Z}_2$-graded. The even (resp. odd) $p$-cocycles
space is defined as $Z_{\alpha,\beta,0}^{p}(\mathcal{A},V)=Z_{\alpha,\beta}^{p}(\mathcal{A},V)\cap C_{\alpha,\beta,0}^{p}(\mathcal{A},V)$
 (resp. $Z_{\alpha,\beta,1}^{p}(\mathcal{A},V)=Z_{\alpha,\beta}^{p}(\mathcal{A},V)\cap C_{\alpha,\beta,1}^{p}(\mathcal{A},V))$.
  \item The $p$-cobord $B_{\alpha,\beta}^{p}(\mathcal{A},V)$ space is $\mathbb{Z}_2$-graded. The even (resp. odd) $p$-cobords
space is defined as $B_{\alpha,\beta,0}^{p}(\mathcal{A},V)=B_{\alpha,\beta}^{p}(\mathcal{A},V)\cap C_{\alpha,\beta,0}^{p}(\mathcal{A},V)$
 (resp. $B_{\alpha,\beta,1}^{p}(\mathcal{A},V)=Z_{\alpha,\beta}^{p}(\mathcal{A},V)\cap B_{\alpha,\beta,1}^{p}(\mathcal{A},V))$.
\end{enumerate}
\end{remark}
\begin{lem}$Z_{\alpha,\beta}^{p}(\mathcal{A},V)\subset B_{\alpha,\beta}^{p}(\mathcal{A},V)$
\end{lem}
\begin{definition}
The $p^{th}$ cohomology space is the quotient $H^p(\mathcal{A},V)=\frac{Z_{\alpha,\beta}^{p}(\mathcal{A},V)}{ B_{\alpha,\beta}^{p}(\mathcal{A},V)}$. It decomposes as
well as even and odd $p^{th}$ cohomology spaces. We denote $H^p(\mathcal{A},V)=H^p_0(\mathcal{A},V)\bigoplus H^p_1(\mathcal{A},V)$.
\end{definition}
\subsection{Second cohomology group of $q$-deformed Heisenberg-Virasoro superalgebra of Hom-type $H^2_0(\mathfrak{A}_q,\mathbb{C})$}
 We denote by $[\varphi]$ the cohomology
class of an element $\varphi$.
\begin{thm}
$H^2_0(\mathfrak{A}_q,\mathbb{C})=\mathbb{C}[\phi]\oplus\mathbb{C}[\varphi]\oplus\mathbb{C}[\psi] $, where
\begin{align}
&\phi(L_n,L_m)=\delta_{n+m,0}\frac{q^{-n}}{6(1+q^n)}\{n+1\}\{n\}\{n-1\},\\
& \phi(L_n,I_m)=\phi(I_n,I_m)=\phi(G_n,G_m)=\phi(G_n,T_m)=\phi(T_n,T_m)=0.\nonumber\\
&\varphi(L_n,I_m)=\delta_{n+m,0}\frac{2q^{-n}}{(1+q^n)}\{n+1\}\{n\},\\
& \varphi(L_n,L_m)=\varphi(I_n,I_m)=\varphi(G_n,G_m)=\varphi(G_n,T_m)=\varphi(T_n,T_m)=0.\nonumber\\
&\psi(I_n,I_m)=\delta_{n+m,0}\frac{2q^{m}}{(1+q^m)}\{n\},\\
& \psi(L_n,L_m)=\psi(I_n,I_m)=\psi(G_n,G_m)=\psi(G_n,T_m)=\psi(T_n,T_m)=0\nonumber.
\end{align}
\end{thm}
\begin{proof}
For all $\varphi\in C^2_{\alpha,Id_{\mathbb{C}}}(\mathfrak{A}_q,\mathbb{C})$, we have
\begin{equation}\label{2cobord}d^2(\varphi)(x_0, x_1, x_2) = -\varphi([x_0, x_1], \alpha(x_2)) + (-1)^{|x_2||x_1|}
\varphi([x_0, x_2], \alpha(x_1)) + \varphi(\alpha(x_0), [x_1, x_2]).\end{equation}
Now, suppose that $\varphi$ is a $q$-deformed $2$-cocycle on $\mathfrak{A}_q$. From \eqref{2cobord}, we obtain
\begin{equation}\label{2cocycle} -\varphi([x_0, x_1], \alpha(x_2)) + (-1)^{|x_2||x_1|}
\varphi([x_0, x_2], \alpha(x_1)) + \varphi(\alpha(x_0), [x_1, x_2])=0.\end{equation}
In \cite{AmmarMakhloufSaadaoui2, Yongsheng CHENGHengyun YANG}
By \eqref{2cocycle} and taking the triple $(x, y, z)$ to be $(L_n,L_m,L_p)$, $(L_n,L_m,I_p)$, $(L_n,I_m,I_p)$, , and $(L_n,G_m,G_p)$, respectively, we obtain $\varphi(L_n,L_p)$, $\varphi(L_n,I_p)$, $\varphi(I_n,I_p)$,  and $\varphi(G_n,G_p)$ which define $\varphi$. Then we have
\begin{align*}
&\varphi(I_n,I_p)=\delta_{n+p,0}\frac{(q+1)\{n+1\}}{q^n+1}\varphi(I_1,I_{-1})\\
&\varphi(L_n,I_p)=\delta_{n+p,0}\big(-\frac{q^{-n}(q+1)(q^n-1)(q^n-q^2)}{(q-1)^2(q^n+1)}\varphi(L_1,I_{-1})
\\&+\frac{q^{1-n}(q^2+1)(q^n-1)(q^n-q)}{(q-1)^2(q+1)(q^n+1)}\varphi(L_1,I_{-1})\big)\\
&\varphi(L_n,L_p)=\delta_{n+p,0}\big(-\frac{A}{B}\varphi(L_1,L_{-1})
+\frac{A'}{B'}\varphi(L_1,L_{-1})\big)\\
&\varphi(G_n,G_p)=0.
\end{align*}
Now taking  the triple $(x, y, z)$ to be  $(L_n,T_m,T_p)$ and $(L_n,T_m,T_p)$ in \eqref{2cocycle} to obtain $\varphi(T_n,T_p)$,  and $\varphi(G_n,T_p)$.\\
$\bullet$ For $(x, y, z)=(L_n,G_m,T_p)$
\begin{equation}\label{2cocycle1} -\varphi([L_m, T_n], \alpha(T_p)) -
\varphi([L_m, T_p], \alpha(T_n)) + \varphi(\alpha(L_m), [T_n, T_p])=0.\end{equation}
Thus
\begin{equation}\label{2cocycle2} -(1+q^{p+1})\{n+1\}\varphi( T_{m+n}, T_p) -
(1+q^{n+1})\{p+1\}\varphi(T_{m+p}, T_n) =0.\end{equation}
Taking $m=0$
$$
 2\frac{q^{p+1}-q^{n+1}}{1-q}\varphi(T_{p}, T_n) =0.$$
 Thus $\varphi(T_n,T_p)=0$ for all $n\neq p.$\\
 Taking $n=0$ and $m=p$ in \eqref{2cocycle2}
$$(1+q^{m+1})\varphi( T_{m}, T_m)= -
(1+q)\{m+1\}\varphi(T_{2m}, T_0) .$$
which implies that  $\varphi(T_m,T_{m})=0$ for all $m\neq0 $.\\
Taking $n=0$ and $m=-p$ in \eqref{2cocycle2}, we have $\varphi(T_0,T_0)=0$. Thus $\varphi(T_n,T_{m})=0$ for all $n,\ m\in\mathbb{Z}.$\\
$\bullet$ For $(x, y, z)=(L_n,G_m,T_p)$
 \begin{equation}\label{2cocycle3} -\varphi([L_m, G_n], \alpha(T_p)) -
\varphi([L_m, T_p], \alpha(G_n)) + \varphi(\alpha(L_m), [G_n, T_p])=0.\end{equation}
Thus
\begin{equation}\label{2cocycle4} (1+q^{p+1})(\{m\}-\{n+1\})\varphi( G_{m+n}, T_p) -
(1+q^{n+1})\{p+1\}\varphi(T_{m+p}, G_n) =0.\end{equation}
Similarly we can prove that $\varphi(G_n,T_{m})=0$ for all $n,\ m\in\mathbb{Z}.$\\
 We denote by $f$ the even linear map defined on $\mathfrak{A}$ by
 \begin{align*}
 &f(L_n)=-\frac{1}{\{n\}}\varphi(L_0,L_n),\ \textrm{if}\ n\neq0,\ f(L_0)=-\frac{q}{1+q}\varphi(L_1,L_{-1}),\\
 &f(I_n)=-\frac{1}{\{n\}}\varphi(L_0,I_n),\ \textrm{if}\ n\neq0,\ f(I_0)=-q\varphi(L_1,I_{-1}),\\
 &f(G_n)=-\frac{1}{\{n+1\}}\varphi(L_0,G_n),\ \textrm{if}\ n\neq-1,\ f(G_{-1})=-\frac{q}{1+q}\varphi(L_1,G_{-2}),\\
 &f(T_n)=-\frac{1}{\{n+1\}}\varphi(L_0,I_n),\ \textrm{if}\ n\neq-1,\ f(T_{-1})=-q\varphi(L_1,T_{-2}).
 \end{align*}
 It is easy to verify that
  \begin{align*}
 &\delta^1(f)(L_n,L_m)=\frac{\{n\}-\{m\}}{\{n+m\}}\varphi(L_0,L_{n+m}),\ \textrm{if}\ n+m\neq0,\ \delta^1(f)(L_n,L_{-n})=0,\\
 &\delta^1(f)(L_n,I_m)=\frac{\{n\}}{\{n+m\}}\varphi(L_0,L_{n+m}),\ \textrm{if}\ n+m\neq0,\ \delta^1(f)(L_n,I_{-n})=0,\\
 &\delta^1(f)(I_n,I_m)=\delta^1(f)(G_n,G_m)=\delta^1(f)(G_n,T_m)=\delta^1(f)(T_n,T_m)=0.
 \end{align*}
\end{proof}

\section{Derivations of the Hom-Lie superalgebra $\mathfrak{A}_q$}
In this section we compute the derivations of $q$-deformed Heisenberg-Virasoro superalgebra of Hom-type $\mathfrak{A}_q$ and its $q$-derivation. A homogeneous $\alpha^{k}$-derivation is said of degree $s$ if there exists $s\in \mathbb{Z}$ such that for all $n\in \mathbb{Z}$ we have
$D(< L_{n} >) \subset < L_{n+s} >$. The corresponding subspace of homogeneous $\alpha^{k}$-derivations of degree $s$
is denoted by $Der^s_{\alpha^k,i}(\mathfrak{A}_q)(i\in\mathbb{Z}_{2})$.\\
It easy to check that
$Der^{s}_{\alpha^{k}}(\mathfrak{A}_q)=\oplus_{s\in \mathbb{Z}}\Big(Der^{s}_{\alpha^{k},0}(\mathfrak{A}_q)\oplus Der^{s}_{\alpha^{k},1}(\mathfrak{A}_q)\Big)$.\\
Let $D$ be a homogeneous $\alpha^{k}$-derivation
$$D([x, y]) = [D(x),\alpha^{k}(y)] + (-1)^{|x||D|}[\alpha^{k}(x),D(y)],\forall~~ x, y\in \mathcal{H}(\mathfrak{A}_q).$$
We deduce that
\begin{equation}\label{first-ega}
(\{m\}-\{n\})D(L_{n+m}) = (1 + q^{m})^{k}[D(L_n),L_m]_{q} + (1 + q^{n})^{k}[L_n,D(L_m)]_{q},
\end{equation}
\begin{equation}\label{second-ega}
   (\{m + 1\}-\{n\})D(I_{n+m}) = (1 + q^{m})^{k}[D(L_n),G_m]_{q} + (1 + q^{n})^{k}[L_n,D(L_m)]_{q},
\end{equation}
\begin{equation}\label{third-egali}
   (\{m + 1\} -\{n\})D(G_{n+m}) = (1 + q^{m})^{k}[D(L_n),G_m]_{q} + (1 + q^{n})^{k}[L_n,D(L_m)]_{q},
\end{equation}
and
\begin{equation}\label{for-egali}
   (\{m + 1\} -\{n\})D(T_{n+m}) = (1 + q^{m})^{k}[D(L_n),G_m]_{q} + (1 + q^{n})^{k}[L_n,D(L_m)]_{q},~~\forall~~n,m\in\mathbb{Z}.
\end{equation}
\subsection{The $\alpha^{0}$-derivation of the Hom-Lie superalgebra $\mathfrak{A}_q$}
\begin{proposition} The set of even  $\alpha^{0}$-derivations of the Hom-Lie superalgebra $\mathfrak{A}_q$ is
$$Der_{\alpha^{0},0}(\mathfrak{A}_q) =< D_1 > \oplus < D_2 >\oplus < D_3 >\oplus < D_4 >$$
where $D_1$, $D_2$, $D_3$ and $D_4$ are defined, with respect to the basis as
\begin{align}&D_{1}(L_n) = nL_n,\ D_{1}(I_n) = -q\{n-1\}L_n+nI_n ,D_{1}(G_n) = 0,\ D_{1}(T_n) =nT_n,\\
&D_2(I_n) =I_n ,\ D_2(L_n) =D_2(G_n) =D_2(T_n) =0,\\
&D_{3}(G_n) = G_n,\ D_{3}(L_n) =D_{3}(L_n) = D_{3}(T_n) =0,\\
&D_4(T_n) =T_n ,\ D_4(L_n) =D_4(G_n) =D_4(I_n) =0.
\end{align}
\end{proposition}
\begin{proof}

Let $D$ ba an even derivation of degree $s$:
\begin{align}
&\label{derivation1}D(L_n) = a_{s,n}L_{s+n}+b_{s,n}I_{s+n},~~D(I_n) = c_{s,n}L_{s+n}+d_{s,n}I_{s+n},\\
&D(G_n) = e_{s,n}G_{s+n}+f_{s,n}T_{s+n},~~D(T_n) = g_{s,n}G_{s+n}+h_{s,n}T_{s+n}\nonumber.
\end{align}
 $a_{s,n}=\delta_{s,0}na_{s,1}$ and $e_{s,n}=\delta_{s,0}e_{s,0}$ (See \cite{AmmarMakhloufSaadaoui1}).\\
By (\ref{first-ega}) and (\ref{derivation1}), we have\\
\begin{eqnarray*}
    (\{m\}-\{n\})b_{s,m+n}&=&\{s+m\}b_{s,m}-\{s+n\}b_{s,n}.
  \end{eqnarray*}
We deduce that
\begin{eqnarray*}
    (q^{n}-q^{m})b_{s,m+n}&=&(1-q^{s+m})b_{s,m}-(1-q^{s+n})b_{s,n}.
  \end{eqnarray*}
If $m = 0$, we have
 \begin{eqnarray*}
    (q^{n}-1)b_{s,n}&=&(1-q^{s})b_{s,0}-(1-q^{s+n})b_{s,n}.
  \end{eqnarray*}
If $s\neq 0$, we have
\begin{eqnarray*}
   b_{s,n}&=&\frac{1}{q^{n}}b_{s,0}.
  \end{eqnarray*}
We deduce that
\begin{eqnarray}
 \label{alfa-0-der2}(q^{n}-q^{m})\frac{1}{q^{n+m}}b_{s,0}&=&(1-q^{s+m})\frac{1}{q^{m}}b_{s,0}+(q^{s+n}-1)\frac{1}{q^{n}}b_{s,0}.
\end{eqnarray}Taking $n=2s$ and $m=2s$ we have $b_{s,0}=0$, so $b_{s,n}=0$ .\\
If $s=0$ then $b_{s,n}=0$.

Similarly we can prove  that
\begin{align*}
c_{s,n}&=-\delta_{s,0}q\{n-2\}c_{s,1},\\ d_{s,n}&=\delta_{s,0}(na_{s,1}+d_{s,1}),\\
f_{s,n}&=0,\\ g_{s,n}&=-\delta_{s,0}q\{n-1\}g_{s,0},\\
 h_{s,n}&=\delta_{s,0}(na_{s,1}+h_{s,0}).
\end{align*}

\begin{proposition} The set of odd  $\alpha^{0}$-derivations of the Hom-Lie superalgebra $\mathfrak{A}_q$ is
$$Der_{\alpha^{0},1}(\mathfrak{A}_q) =< D_1 > \oplus < D_2 >\oplus < D_3 >\oplus < D_4 >$$
where $D_1$, $D_2$, $D_3$ and $D_4$ are defined, with respect to the basis as
\begin{align}&D_{1}(L_n) = nG_{n-1},\ D_{1}(I_n) = nT_{n-1} ,D_{1}(G_n) =  D_{1}(T_n) =0,\\
&D_2(G_n) =\frac{q^n-1}{q^n}I_{n+1} ,\ D_2(L_n) =D_2(I_n) =D_2(T_n) =0,\\
&D_{3}(T_n) =-q\{n-1\} L_{n+1},\ D_{3}(L_n) =D_{3}(I_n) = D_{3}(G_n) =0,\\
&D_4(T_n) =I_{n+1} ,\ D_4(L_n) =D_4(G_n) =D_4(I_n) =0.
\end{align}
\end{proposition}
Let $D$ be an odd derivation of degree $s$:
\begin{align}\label{derivation2}
&D(L_n) = a_{s,n}G_{s+n}+b_{s,n}T_{s+n},~~D(I_n) = c_{s,n}G_{s+n}+d_{s,n}T_{s+n},\\
&D(G_n) = e_{s,n}L_{s+n}+f_{s,n}I_{s+n},~~ D(T_n) = g_{s,n}L_{s+n}+h_{s,n}I_{s+n}.\nonumber
\end{align}
 $a_{s,n}=\delta_{s,-1}na_{s,1}$ and $e_{s,n}=0$ (See  \cite{AmmarMakhloufSaadaoui1}).\\
By \eqref{first-ega} and (\ref{derivation2}) we have
\begin{eqnarray*}
    (\{m\}-\{n\})b_{s,m+n}&=&\{s+m+\}b_{s,m}-\{s+n+1\}b_{s,n}.
  \end{eqnarray*}
  We deduce that
\begin{eqnarray*}
    (q^{n}-q^{m})b_{s,m+n}&=&(1-q^{s+m+1})b_{s,m}-(1-q^{s+n+1})b_{s,n}.
  \end{eqnarray*}
If $m = 0$, we have
 \begin{eqnarray*}
    (q^{n}-1)b_{s,n}&=&(1-q^{s+1})b_{s,0}-(1-q^{s+n+1})b_{s,n}.
  \end{eqnarray*}
If $s\neq -1$, we have
\begin{eqnarray*}
   b_{s,n}&=&\frac{1}{q^{n}}b_{s,0}.
  \end{eqnarray*}
We deduce that
\begin{equation}\label{alfa-1-der1}
 (q^{n}-q^{m})\frac{1}{q^{n+m}}b_{s,0}=(1-q^{s+m+1})\frac{1}{q^{m}}b_{s,0}+(q^{s+n+1}-1)\frac{1}{q^{n}}b_{s,0}
\end{equation}
Taking $n=2s+2$ and $m=s+1$, we have $b_{s,0}=0$, so $b_{s,n}=0$.\\
If $s=-1$ and $n\neq m$, we have
\begin{eqnarray*}
    b_{-1,m+n}&=&\frac{1-q^{m}}{q^{n}-q^{m}}b_{-1,m}-\frac{1-q^{n}}{q^{n}-q^{m}}b_{-1,n}.
  \end{eqnarray*}
  If $m=1, \ n=4$, then
\begin{eqnarray}\label{14}
    b_{-1,5}&=&\frac{1-q}{q^{4}-q}b_{-1,1}-\frac{1-q^{4}}{q^{4}-q}b_{-1,4}.
  \end{eqnarray}If $m=1, \ n=3$, then
\begin{eqnarray}\label{13}
    b_{-1,4}&=&\frac{1-q}{q^{3}-q}b_{-1,1}-\frac{1-q^{3}}{q^{3}-q}b_{-1,3}.
  \end{eqnarray}If $m=1, \ n=2$, then
\begin{eqnarray}\label{12}
    b_{-1,3}&=&\frac{1-q}{q^{2}-q}b_{-1,1}-\frac{1-q^{2}}{q^{2}-q}b_{-1,2}\\
    &=&\frac{-1}{q}b_{-1,1}+\frac{1+q}{q}b_{-1,2}.\nonumber
  \end{eqnarray}
  Thus, using \eqref{14}, \eqref{13} and \eqref{12}, we have
  \begin{eqnarray*}
    b_{-1,5}&=&\frac{1-q}{q^{4}-q}b_{-1,1}-\frac{1-q^{4}}{q^{4}-q}b_{-1,4}\\
&=&\frac{1-q}{q^{4}-q}b_{-1,1}-\frac{1-q^{4}}{q^{4}-q}(\frac{1-q}{q^{3}-q}b_{-1,1}-\frac{1-q^{3}}{q^{3}-q}b_{-1,3})\\
&\vdots&\\
&=&0.
  \end{eqnarray*}
  So $b_{-1,n}=0$, then, for all $s,\ n$, $b_{s,n}=0$.\\
Similarly we can prove that
\begin{align*}
c_{s,n}&=0,\\ d_{s,n}&=\delta_{s,-1}na_{s,1},\\
f_{s,n}&=\delta_{s,1}\frac{q^n-1}{q^{n}}f_{s,0},\\ g_{s,n}&=-\delta_{s,1}q\{n-1\}g_{s,0},\\
 h_{s,n}&=\delta_{s,1}h_{s,0}.
\end{align*}
\end{proof}
\subsection{The $\alpha^{1}$-derivation of  Hom-Lie superalgebra $\mathfrak{A}_q$}
\begin{proposition}If $D$ is an $\alpha$-derivation then $D = 0$.
\end{proposition}
\begin{proof}\textbf{Case 1}$:~~ |D| = 0$\\
Let $D$ be an even derivation of degree $s$:
\begin{align}\label{alpha-deriv1}
&D(L_n) = a_{s,n}L_{s+n}+b_{s,n}I_{s+n},~~D(I_n) = c_{s,n}L_{s+n}+d_{s,n}I_{s+n},\\&D(G_n) = e_{s,n}G_{s+n}+f_{s,n}T_{s+n},~~
D(T_n) = g_{s,n}G_{s+n}+h_{s,n}T_{s+n}.\nonumber\end{align}
By (\ref{first-ega}) and (\ref{alpha-deriv1}) we have
\begin{eqnarray*}
   (\{m\}-\{n\})a_{s,n+m} &=& (1 + q^{m})(\{m\}-\{s+n\})a_{s,n} + (1 + q^{n})(\{s+m\}-\{n\})a_{s,m}, \\
    (\{m\}-\{n\})b_{s,n+m} &=& (1 + q^{n})\{s+m\}a_{s,m} -(1 + q^{m})\{s+n\}b_{s,n}.
 \end{eqnarray*}
We deduce that
\begin{eqnarray*}
   (q^{n}-q^{m})a_{s,n+m} &=& (1 + q^{m})(q^{s+n}-q^{m})a_{s,n} + (1 + q^{n})(q^{n}-q^{s+m})a_{s,m}, \\
    (q^{n}-q^{m})b_{s,n+m} &=& (1 + q^{n})(1-q^{s+m})a_{s,m} -(1 + q^{m})(1-q^{s+n})b_{s,n}.
  \end{eqnarray*}
Then $a_{s,n} = 0$ and $b_{s,n}=0$, so $D(L_n) = 0$.\\
By (\ref{second-ega}) and (\ref{alpha-deriv1}) we have
\begin{eqnarray*}
   \{n\}c_{s,n+m} &=& (1 + q^{m})(\{n\}-\{m\})c_{s,n}, \\
    \{n\}d_{s,n+m} &=& (1 + q^{m})\{s+n\}d_{s,m} +(1 + q^{n})\{n\}a_{s,m}.
  \end{eqnarray*}
We deduce that
\begin{eqnarray*}
   (1-q^{n})c_{s,n+m} &=& (1 + q^{m})(q^{m}-q^{n})c_{s,n}, \\
    (1-q^{n})d_{s,n+m} &=& (1 + q^{m})(1-q^{s+n})d_{s,m} +(1 + q^{n})(1-q^{n})a_{s,m}.
  \end{eqnarray*}
Then $c_{s,n} = 0$ and $d_{s,n}=0$, so $D(I_n) = 0$.\\
Similarly we have $D(G_n) =D(T_n) = 0$. Hence $D\equiv 0$. \\
\textbf{Case 2}$:~~ |D| = 1$\\
Let $D$ be an odd derivation of degree $s$:
\begin{align}\label{alpha-deri2}
&D(L_n) = a_{s,n}G_{s+n}+b_{s,n}T_{s+n},~~D(I_n) = c_{s,n}G_{s+n}+d_{s,n}T_{s+n},\\&D(G_n) = e_{s,n}L_{s+n}+f_{s,n}I_{s+n},~~D(T_n) = g_{s,n}L_{s+n}+h_{s,n}I_{s+n}.\nonumber
\end{align}

By (\ref{first-ega}) and (\ref{alpha-deri2}) we have
\begin{eqnarray*}
   (\{m\}-\{n\})a_{s,m+n}&=&(1+q^{n})(\{s+m+1\}-\{n\})a_{s,m}+(1+q^{m})(\{m\}-\{s+n+1\})a_{s,n},\\
    (\{m\}-\{n\})b_{s,m+n}&=&(1+q^{n})\{s+m+1\}b_{s,m}-(1+q^{m})\{s+n+1\}b_{s,n}.
  \end{eqnarray*}
Then if $m\neq n$, we have
\begin{eqnarray*}
   a_{s,m+n}&=&\frac{(1+q^{n})(q^{n}-q^{s+m+1})}{q^{n}-q^{m}}a_{s,m}-\frac{(1+q^{m})(q^{m}-q^{s+n+1})}{q^{n}-q^{m}}a_{s,n}, \\
    b_{s,m+n}&=&\frac{(1+q^{n})(1-q^{s+m+1})}{q^{n}-q^{m}}b_{s,m}-\frac{(1+q^{m})(1-q^{s+n+1})}{q^{n}-q^{m}}b_{s,n}.
  \end{eqnarray*}
If $m=0$, we have
\begin{eqnarray*}
   a_{s,n}&=& \frac{(1+q^{n})(q^{n}-q^{s+1})}{1+q^{n}-2q^{s+n+1}}a_{s,0},\\
    b_{s,n}&=& \frac{(1+q^{n})(1-q^{s+1})}{2q^{s+n+1}-1}b_{s,0}.
  \end{eqnarray*}
So
\begin{eqnarray*}
   a_{s,m+n}&=& \frac{(1+q^{m+n})(q^{m+n}-q^{s+1})}{1+q^{m+n}-2q^{s+m+n+1}}a_{s,0},\\
    b_{s,m+n}&=& \frac{(1+q^{m+n})(1-q^{s+1})}{2q^{s+m+n+1}-1}b_{s,0}.
  \end{eqnarray*}
Then
\begin{eqnarray*}
   \frac{(1+q^{m+n})(q^{m+n}-q^{s+1})}{1+q^{m+n}-2q^{s+m+n+1}}a_{s,0}&=&\frac{(1+q^{n})(q^{n}-q^{s+m+1})(1+q^{m})(q^{m}-q^{s+1})}
{(q^{n}-q^{m})(1+q^{m+n}-2q^{s+m+n+1})}a_{s,0}  \\
   &-&\frac{(1+q^{m})(q^{m}-q^{s+n+1})(1+q^{n})(q^{n}-q^{s+1})}{(q^{n}-q^{m})(1+q^{n}-2q^{s+n+1})}a_{s,0}
\end{eqnarray*}
and
\begin{eqnarray*}
   \frac{(1+q^{m+n})(1-q^{s+1})}{2q^{s+m+n}}b_{s,0}&=&\frac{(1+q^{n})(1-q^{s+m+1})(1+q^{m})(1-q^{s+1})}
{(q^{n}-q^{m})(2q^{s+m+1}-1)}b_{s,0}  \\
   &-&\frac{(1+q^{m})(1-q^{s+n+1})(1+q^{n})(1-q^{s+1})}{(q^{n}-q^{m})(2q^{s+n+1}-1)}b_{s,0}.
\end{eqnarray*}
If $q \in [0,1[$, then letting $n,m\longrightarrow \propto $, we obtain $a_{s,0} = 0$ and $b_{s,0} = 0$. If $q > 1$ and setting $m = s$,
then if $n$ goes to infinity we obtain $a_{s,0}=0$ and $b_{s,0} = 0$. We deduce that $D(L_n)=0$.\\
By (\ref{second-ega}) and $D(I_n) = c_{s,n}G_{s+n}+d_{s,n}T_{s+n}$, we obtain
\begin{eqnarray*}
   \{n\}c_{s,m+n}&=&(1+q^{m})(\{s+n+1\}-\{m\})c_{s,n},  \\
    \{n\}d_{s,m+n}&=&(1+q^{m})\{s+n+1\}d_{s,n}+(1+q^{n})\{n\}a_{s,m}.
  \end{eqnarray*}
If $n\neq 0$, we have
\begin{eqnarray*}
   c_{s,m+n}&=&\frac{(1+q^{m})(q^{m}-q^{s+n+1})}{1-q^{n}}c_{s,n},  \\
    d_{s,m+n}&=&\frac{(1+q^{m})(1-q^{s+n+1})}{1-q^{n}}d_{s,n}+(1+q^{n})a_{s,m}.
  \end{eqnarray*}
If $m=0$, we have
\begin{eqnarray*}
   (2q^{s+n+1}-q^{n}-1)c_{s,n}&=&0,\\
    d_{s,n}&=&\frac{2(1-q^{s+n+1})}{1-q^{n}}d_{s,n}+(1+q^{n})a_{s,0}.
\end{eqnarray*}
Since $a_{s,0}=0$, so $c_{s,n}=0$ and $d_{s,n}=0$. We deduce that $D(I_n)=0$.\\
By (\ref{third-egali}) and (\ref{alpha-deri2}), we
deduce that $D(G_n)=0$ and by (\ref{for-egali}) and (\ref{alpha-deri2}), we have
$D(T_n)=0$, hence $D\equiv 0.$
\end{proof}

\subsection{The $q$-derivation of the Hom-Lie superalgebra $\mathfrak{A}_q$}
In this section, we study the $q$-derivations of $\mathfrak{A}_q$. The derivation algebra of $\mathfrak{A}_q$ is denoted by
$\mathcal{Q}Der\mathfrak{A}_q$. Since $\mathfrak{A}_q$ is $\mathbb{Z}_{2}$-graded Hom-Lie superalgebra, we have
$$\mathcal{Q}Der\mathfrak{A}_q=\oplus_{s\in \mathbb{Z}}(\mathcal{Q}Der\mathfrak{A}_q)_{0,s}\oplus (\mathcal{Q}Der\mathfrak{A}_q)_{1,s}, $$
where $(\mathcal{Q}Der\mathfrak{A}_q)_{0}$ denotes the set of even derivations of $\mathfrak{A}_q$, and $(\mathcal{Q}Der\mathfrak{A}_q)_{1}$ denotes the set of odd derivations of $\mathfrak{A}_q$
\begin{definition} Let  $\varphi :\mathfrak{A}_q\rightarrow\mathfrak{A}_q$ a linear map, then $\varphi$ is called an even  $q$-derivation (resp. an odd  $q$-derivation) if
\begin{equation}\label{q-deri-paire}
\varphi([x,y]_{q})=\frac{1}{1+q^{s}}\Big([\varphi(x),\alpha_{q}(y)]_{q}+[\alpha_{q}(x),\varphi(y)]_{q}\Big)
\end{equation}
\begin{equation}\label{q-deri-impaire}
(resp.~~ \varphi([x,y]_{q})=\frac{1}{1+q^{s+1}}\Big([\varphi_{q}(x),\alpha_{q}(y)]_{q}+(-1)^{|x|}[\alpha_{q}(x),\varphi(y)]_{q})\Big).
\end{equation}
where $x, y$ are homogeneous elements in $ \mathfrak{A}_q$.\\
For a fixed $a\in (\mathfrak{A}_q)_{i}$, we obtain the following $q$-derivation
$$\begin{array}{cccc}
  \varphi: & \mathfrak{A}_q & \longrightarrow & \mathfrak{A}_q \\
   & x & \longmapsto & [a,x].
\end{array}$$
The map is denoted by $ad_{a}$ and is called the inner $q$-derivation.
\end{definition}
\begin{proposition}If $\varphi$ is an even $q$-derivation of degree $s\neq0$ then it is an inner derivation, more
precisely:
$$(\mathcal{Q}Der\mathfrak{A}_q)_{0,s} =<ad_{L_{s}}+ad_{I_{s}}> .$$
\end{proposition}
\begin{proof}Let $\varphi$ be an even $q$-derivation of degree  $s\neq 0$:
\begin{align}\label{even-q-der}
&\varphi(L_n) = a_{s,n}L_{s+n}+b_{s,n}I_{s+n},~~\varphi(I_n) = c_{s,n}L_{s+n}+d_{s,n}I_{s+n},\\&
\varphi(G_n) = e_{s,n}G_{s+n}+f_{s,n}T_{s+n},~~\varphi(T_n) = g_{s,n}G_{s+n}+h_{s,n}T_{s+n}.\nonumber
\end{align}

By (\ref{crochet1}) and (\ref{even-q-der}), we have
\begin{eqnarray*}
  \{n\}\varphi(L_n) &=&\varphi([L_n,L_0]_{q})\\
   &=& \frac{1}{1+q^{s}}\Big([\varphi(L_n),\alpha_{q}(L_0)]_{q}+[\alpha_{q}(L_n),\varphi(L_0)]_{q} \Big)\\
   &=&\frac{1}{1+q^{s}}\Big([a_{s,n}L_{s+n}+b_{s,n}I_{s+n},2L_0]_{q}+[(1+q^{n})L_n,a_{s,0}L_{s}+b_{s,0}I_{s}]_{q} \Big)\\
   &=&\Big( \frac{1+q^{n}}{1+q^{s}}(\{n\}-\{s\})a_{s,0}+ \frac{2}{1+q^{s}} \{s+n\}a_{s,n}\Big) L_{s+n}
+\Big(\frac{2}{1+q^{s}}\{s+n\}b_{s,n}-\frac{1+q^{n}}{1+q^{s}}\{s\}b_{s,0}\Big)I_{s+n},
\end{eqnarray*}
then \begin{eqnarray*}
      \{n\}a_{s,n}&=& \frac{1+q^{n}}{1+q^{s}}(\{n\}-\{s\})a_{s,0}+ \frac{2}{1+q^{s}} \{s+n\}a_{s,n},\\
       \{n\}b_{s,n}&=& \frac{2}{1+q^{s}}\{s+n\}b_{s,n}-\frac{1+q^{n}}{1+q^{s}}\{s\}b_{s,0}.
       \end{eqnarray*}
We deduce that, $a_{s,n}=\frac{q^{s}-q^{n}}{q^{s}-1}a_{s,0}$ and $b_{s,n}=b_{s,0}$. Moreover,
\begin{eqnarray*}
\frac{a_{s,0}}{\{s\}}ad_{L_s}(L_n)+ \frac{b_{s,0}}{\{s\}}ad_{I_s}(L_n)&=&\frac{a_{s,0}}{\{s\}}[L_s,L_n]_{q}+\frac{b_{s,0}}{\{s\}}[I_s,L_n]_{q}\\
&=&\frac{a_{s,0}}{\{s\}}(\{s\}-\{n\})L_{s+n}+\frac{b_{s,0}}{\{s\}}\{s\}I_{s+n}\\
&=&\frac{q^{n}-q^{s}}{1-q^{s}}a_{s,0}L_{s+n}+\frac{b_{s,0}}{1-q^{s}}I_{s+n}\\
&=& a_{s,n}L_{s+n}+b_{s,n}I_{s+n}.
\end{eqnarray*}
So
\begin{equation}\label{q-deriv-paire-ad1}
    \varphi(L_n)=\frac{a_{s,0}}{\{s\}}ad_{L_s}(L_n)+ \frac{b_{s,0}}{\{s\}}ad_{I_s}(L_n).
\end{equation}
By (\ref{crochet2}) and (\ref{even-q-der}), we obtain
\begin{eqnarray*}
  \{n\}\varphi(I_n) &=&\varphi([I_n,L_0]_{q})  \\
   &=& \frac{1}{1+q^{s}}\Big([\varphi(I_n),\alpha_{q}(L_0)]_{q}+[\alpha_{q}(I_n),\varphi(L_0)]_{q} \Big)\\
   &=&\frac{1}{1+q^{s}}\Big([c_{s,n}L_{s+n}+d_{s,n}I_{s+n},2L_0]_{q}+[(1+q^{n})I_n,a_{s,0}L_{s}+b_{s,0}I_{s}]_{q} \Big)\\
   &=& \frac{2}{1+q^{s}} \{s+n\}c_{s,n} L_{s+n}
+\Big(\frac{2}{1+q^{s}}\{s+n\}d_{s,n}+\frac{1+q^{n}}{1+q^{s}}\{n\}a_{s,0}\Big)I_{s+n},
\end{eqnarray*}
then \begin{eqnarray*}
      \{n\}c_{s,n}&=& \frac{2}{1+q^{s}} \{s+n\}c_{s,n},\\
       \{n\}d_{s,n}&=& \frac{2}{1+q^{s}}\{s+n\}d_{s,n}+\frac{1+q^{n}}{1+q^{s}}\{n\}a_{s,0}.
       \end{eqnarray*}
We deduce that, $c_{s,n}=0,~~\forall~~n\in \mathbb{N}$ and $d_{s,n}=\frac{1-q^{n}}{q^{s}-1}a_{s,0}$. Moreover,
\begin{eqnarray*}
\frac{a_{s,0}}{\{s\}}ad_{L_s}(I_n)+ \frac{b_{s,0}}{\{s\}}ad_{I_s}(I_n)&=&\frac{a_{s,0}}{\{s\}}[L_s,L_n]_{q}+\frac{b_{s,0}}{\{s\}}[I_s,L_n]_{q}\\
&=&\frac{-\{n\}}{\{s\}}a_{s,0}L_{s+n}\\
&=&-\frac{1-q^{n}}{1-q^{s}}a_{s,0}I_{s+n}\\
&=& d_{s,n}I_{s+n}.
\end{eqnarray*}
So
\begin{equation}\label{q-deriv-paire-ad2}
    \varphi(I_n)=\frac{a_{s,0}}{\{s\}}ad_{L_s}(I_n)+ \frac{b_{s,0}}{\{s\}}ad_{I_s}(I_n).
\end{equation}
Applying the same relations (\ref{crochet3}) and (\ref{even-q-der}), we have
\begin{equation}\label{q-deriv-paire-ad3}
    \varphi(G_n)=\frac{a_{s,0}}{\{s\}}ad_{L_s}(G_n)+ \frac{b_{s,0}}{\{s\}}ad_{I_s}(G_n).
\end{equation}

And by (\ref{crochet5}) and (\ref{even-q-der}), we have
\begin{equation}\label{q-deriv-paire-ad4}
    \varphi(T_n)=\frac{a_{s,0}}{\{s\}}ad_{L_s}(T_n)+ \frac{b_{s,0}}{\{s\}}ad_{I_s}(T_n).
\end{equation}
Using (\ref{q-deriv-paire-ad1}), (\ref{q-deriv-paire-ad2}), (\ref{q-deriv-paire-ad3}) and (\ref{q-deriv-paire-ad4}) we deduce that $\varphi=\frac{a_{s,0}}{\{s\}}ad_{L_s}+ \frac{b_{s,0}}{\{s\}}ad_{I_s}$.\\
\end{proof}
\begin{proposition}If $\varphi$ is an odd $q$-derivation of degree $s\neq-1$ then it is an inner derivation, more
precisely:
$$(\mathcal{Q}Der\mathfrak{A}_q)_{1,s} =<ad_{G_{s}}+ad_{T_{s}}>.$$
\end{proposition}
\begin{proof}Let $\varphi$ be an odd $q$-derivation of degree $s\neq -1$:
\begin{align}\label{odd-q-derivation}
&\varphi(L_n) = a_{s,n}G_{s+n}+b_{s,n}T_{s+n},~~\varphi(I_n) = c_{s,n}G_{s+n}+d_{s,n}T_{s+n},\\&
\varphi(G_n) = e_{s,n}L_{s+n}+f_{s,n}I_{s+n},~~\varphi(T_n) = g_{s,n}L_{s+n}+h_{s,n}I_{s+n},\nonumber
\end{align}
By (\ref{crochet1}) and (\ref{odd-q-derivation}), we have
\begin{eqnarray*}
  \{n\}\varphi(L_n) &=&\varphi([L_n,L_0]_{q})  \\
   &=& \frac{1}{1+q^{s+1}}\Big([\varphi(L_n),\alpha_{q}(L_0)]_{q}+[\alpha_{q}(L_n),\varphi(L_0)]_{q} \Big)\\
   &=&\frac{1}{1+q^{s+1}}\Big([a_{s,n}G_{s+n}+b_{s,n}T_{s+n},2L_0]_{q}+[(1+q^{n})L_n,a_{s,0}G_{s}+b_{s,0}T_{s}]_{q} \Big)\\
   &=&\Big(\frac{1+q^{n}}{1+q^{s+1}}(\{n\}-\{s+1\})a_{s,0}+ \frac{2}{1+q^{s+1}}\{s+n+1\}a_{s,n}\Big) G_{s+n}\\
&+& \Big(\frac{2}{1+q^{s+1}}\{s+n\}b_{s,n}-\frac{2(1+q^{n})}{1+q^{s+1}}\{s+1\}b_{s,0}\Big)T_{s+n},
\end{eqnarray*}
then
\begin{eqnarray*}
    \{n\}a_{s,n}&=& \frac{1+q^{n}}{1+q^{s+1}}(\{n\}-\{s+1\})a_{s,0}+ \frac{2}{1+q^{s+1}}\{s+n+1\}a_{s,n}, \\
    \{n\}b_{s,n}&=& \frac{2}{1+q^{s+1}}\{s+n\}b_{s,n}-\frac{2(1+q^{n})}{1+q^{s+1}}\{s+1\}b_{s,0}.
  \end{eqnarray*}
We deduce that, $a_{s,n} = \frac{q^{s+1}-q^{n}}{q^{s+1}-1} a_{s,0}$ and $b_{s,n}=b_{s,0}$. On the other hand,\\
\begin{eqnarray*}
\frac{a_{s,0}}{\{s+1\}}ad_{G_s}(L_n)+\frac{b_{s,0}}{\{s+1\}}ad_{T_s}(L_n)&=&\frac{a_{s,0}}{\{s+1\}}[G_s,L_n]_{q}+\frac{b_{s,0}}{\{s+1\}}[T_s,L_n]_{q}\\
&=&\frac{a_{s,0}}{\{s+1\}}(\{s+1\}-\{n\})G_{s+n}+\frac{b_{s,0}}{\{s+1\}}\{s+1\}T_{s+n}\\
&=& \frac{q^{s+1}-q^{n}}{q^{s+1}-1} a_{s,0}G_{s+n}+b_{s,0}T_{s+n}\\
&=& a_{s,n}G_{s,n}+b_{s,n}T_{s,n}.
\end{eqnarray*}
So $\varphi(L_n)=\frac{a_{s,0}}{\{s+1\}}ad_{G_{s}}(L_n)+\frac{b_{s,0}}{\{s+1\}}ad_{T_s}(L_n)$.\\
By (\ref{crochet2}) and (\ref{odd-q-derivation}), we have
\begin{eqnarray*}
  \{n\}\varphi(I_n) &=&\varphi([I_n,L_0]_{q})  \\
   &=& \frac{1}{1+q^{s+1}}\Big([\varphi(I_n),\alpha_{q}(L_0)]_{q}+[\alpha_{q}(I_n),\varphi(L_0)]_{q} \Big)\\
   &=&\frac{1}{1+q^{s+1}}\Big([c_{s,n}G_{s+n}+d_{s,n}T_{s+n},2L_0]_{q}+[(1+q^{n})I_n,a_{s,0}G_{s}+b_{s,0}T_{s}]_{q} \Big)\\
   &=& \frac{2}{1+q^{s+1}}\{s+n+1\}c_{s,n}G_{s+n}
+ \Big(\frac{2}{1+q^{s+1}}\{s+n+1\}d_{s,n}+\frac{1+q^{n}}{1+q^{s+1}}\{n\}b_{s,0}\Big)T_{s+n},
\end{eqnarray*}
then
\begin{eqnarray*}
    \{n\}c_{s,n}&=& \frac{2}{1+q^{s+1}}\{s+n+1\}c_{s,n},\\
    \{n\}d_{s,n}&=& \frac{2}{1+q^{s+1}}\{s+n+1\}d_{s,n}+\frac{1+q^{n}}{1+q^{s+1}}\{n\}a_{s,0}.
  \end{eqnarray*}
We deduce that, $c_{s,n} =0$ and $d_{s,n}= \frac{1-q^{n}}{q^{s+1}-1}a_{s,0}$. On the other hand,\\
\begin{eqnarray*}
\frac{a_{s,0}}{\{s+1\}}ad_{G_s}(I_n)+\frac{b_{s,0}}{\{s+1\}}ad_{T_s}(I_n)&=&\frac{a_{s,0}}{\{s+1\}}[G_s,I_n]_{q}+\frac{b_{s,0}}{\{s+1\}}[T_s,I_n]_{q}\\
&=&-\frac{\{n\}}{\{s+1\}}a_{s,0}T_{s+n}\\
&=& \frac{q^{n}-1}{1-q^{s+1}} a_{s,0}T_{s+n}\\
&=& d_{s,n}T_{s,n}.
\end{eqnarray*}
So $\varphi(I_n)=\frac{a_{s,0}}{\{s+1\}}ad_{G_{s}}(I_n)+\frac{b_{s,0}}{\{s+1\}}ad_{T_s}(I_n)$.\\
similarly, by (\ref{crochet3}) and (\ref{odd-q-derivation}), we have
 $$\varphi(G_n)=\frac{a_{s,0}}{\{s+1\}}ad_{G_{s}}(G_n)+\frac{b_{s,0}}{\{s+1\}}ad_{T_s}(G_n).$$
And by (\ref{crochet5}) and (\ref{odd-q-derivation}), we have
 $$\varphi(T_n)=\frac{a_{s,0}}{\{s+1\}}ad_{G_{s}}(T_n)+\frac{b_{s,0}}{\{s+1\}}ad_{T_s}(T_n).$$
Which implies that $$\varphi=\frac{a_{s,0}}{\{s+1\}}ad_{G_{s}}+\frac{b_{s,0}}{\{s+1\}}ad_{T_s}.$$
\end{proof}

\section{Cohomology of $n$-ary Hom-superalgebras induced by cohomology of Hom-Leibniz superalgebras}

\subsection{$n$-ary Hom-Nambu superalgebras}
In this section, we recall the definitions of $n$-ary Hom-Nambu algebras and $n$-ary Hom-Nambu-Lie algebras, introduced in \cite{makh} by Ataguema, Makhlouf and Silvestrov and we generalize them  to superalgebras cases.
\begin{definition}
An \emph{$n$-ary Hom-Nambu} algebra is a triple $(N, [\cdot ,..., \cdot],  \widetilde{\alpha} )$ consisting of a vector space  $N$, an
$n$-linear map $[\cdot ,..., \cdot ] :  N^{ n}\longrightarrow N$ and a family
$\widetilde{\alpha}=(\alpha_i)_{1\leq i\leq n-1}$ of  linear maps $ \alpha_i:\ \ N\longrightarrow N$, satisfying \\
  \begin{eqnarray}\label{NambuIdentity}
  && \big[\alpha_1(x_1),....,\alpha_{n-1}(x_{n-1}),[y_1,....,y_{n}]\big]= \\ \nonumber
&& \sum_{i=1}^{n}\big[\alpha_1(y_1),....,\alpha_{i-1}(y_{i-1}),[x_1,....,x_{n-1},y_i]
  ,\alpha_i(y_{i+1}),...,\alpha_{n-1}(y_n)\big],
  \end{eqnarray}
  for all $(x_1,..., x_{n-1})\in N^{ n-1}$, $(y_1,...,  y_n)\in N^{ n}.$\\
  The identity $\eqref{NambuIdentity}$ is called \emph{Hom-Nambu identity}.
  \end{definition}
  \begin{definition}
An \emph{$n$-ary Hom-Nambu} superalgebra is a triple $(N, [\cdot ,..., \cdot],  \widetilde{\alpha} )$ consisting of a vector space  $N=N_0\oplus N_1$, an even
$n$-linear map $[\cdot ,..., \cdot ] :  N^{ n}\longrightarrow N$ such that $[N_{j_1} ,..., N_{j_{n}} ]\subset N_{j_1+...+j_{n}}$ and a family
$\widetilde{\alpha}=(\alpha_i)_{1\leq i\leq n-1}$ of even linear maps $ \alpha_i:\ \ N\longrightarrow N$, satisfying \\
  \begin{eqnarray}\label{SNambuIdentity}
  && \big[\alpha_1(x_1),....,\alpha_{n-1}(x_{n-1}),[y_1,....,y_{n}]\big]= \\ \nonumber
&& \sum_{i=1}^{n}(-1)^{|y||x^i|}\big[\alpha_1(y_1),....,\alpha_{i-1}(y_{i-1}),[x_1,....,x_{n-1},y_i]
  ,\alpha_i(y_{i+1}),...,\alpha_{n-1}(y_n)\big],
  \end{eqnarray}
  for all $(x_1,..., x_{n-1})\in \mathcal{H}(N)^{ n-1}$, $(y_1,...,  y_n)\in \mathcal{H}(N)^{ n},$ where $|x^i|=\sum\limits_{j=1}^{i-1}|x_j|$ and $|y|=\sum\limits_{j=1}^{n-1}|y_j|$\\
  The identity $\eqref{SNambuIdentity}$ is called \emph{Super-Hom-Nambu identity}.
  \end{definition}
Let
$x=(x_1,\ldots,x_{n-1})\in \mathcal{H}(N)^{n-1}$, $\widetilde{\alpha}
(x)=(\alpha_1(x_1),\ldots,\alpha_{n-1}(x_{n-1}))\in \mathcal{H}(N)^{n-1}$ and
$y\in \mathcal{H}(N)$. We define an adjoint map  $ad(x)$ as  a linear map on $N$,
such that
\begin{equation}\label{adjointMapNaire}
L_x(y)=[x_{1},\cdots,x_{n-1},y].
\end{equation}

Then the Super-Hom-Nambu identity \eqref{NambuIdentity} may be written in terms of adjoint map as
\begin{equation*}
L_{\widetilde{\alpha} (x)}( [y_1,...,y_n])=
\sum_{i=1}^{n}(-1)^{|y||x^i|}{[\alpha_1(y_1),...,\alpha_{i-1}(y_{i-1}),
L_x(y_{i}), \alpha_{i+1}(y_{i+1}) ...,\alpha_{n-1}(y_{n})].}
\end{equation*}

\begin{remark}
When the maps $(\alpha_i)_{1\leq i\leq n-1}$ are all identity maps, one recovers the classical $n$-ary Nambu superalgebras. The Super-Hom-Nambu Identity \eqref{NambuIdentity}, for $n=2$,  corresponds to Super-Hom-Jacobi identity (see \cite{MS}), which reduces to Super-Jacobi identity when $\alpha_1=id$.
\end{remark}

\begin{definition}
An $n$-ary Hom-Nambu superalgebra $(N, [\cdot ,..., \cdot],  \widetilde{ \alpha} )$ where  $\widetilde{\alpha}=(\alpha_i)_{1\leq i\leq n-1}$
is called \emph{$n$-ary Hom-Nambu-Lie} superalgebra if the bracket is super skew-symmetric that is
\begin{equation}
[x_1,...,x_i,....,x_j,...,x_n]=-(-1)^{|x_i||x_j|}[x_1,...,x_j,....,x_i,...,x_n],\ \ \forall\ i,j\in \{1,...,n\}
\ \ \textrm{and}\ \ \forall\ x_1,...,x_n\in \mathcal{H}(N)
\end{equation}
\end{definition}
In the sequel we deal with a particular class of $n$-ary Hom-Nambu-Lie superalgebras which we call $n$-ary multiplicative Hom-Nambu-Lie superalgebras.
\begin{definition}
An \emph{$n$-ary multiplicative Hom-Nambu superalgebra }
 is an $n$-ary Hom-Nambu superalgebra   $(N, [\cdot ,..., \cdot],  \widetilde{ \alpha})$ with  $\widetilde{\alpha}=(\alpha_i)_{1\leq i\leq n-1}$
where  $\alpha_1=...=\alpha_{n-1}=\alpha$  and satisfying
\begin{equation}
\alpha([x_1,..,x_n])=[\alpha(x_1),..,\alpha(x_n)],\ \  \forall\ x_1,...,x_n\in \mathcal{H}(N).
\end{equation}
For simplicity, we will denote the $n$-ary multiplicative Hom-Nambu superalgebra as $(N, [\cdot ,..., \cdot ],  \alpha)$ where $\alpha :N\rightarrow N$ is a linear map. Also by misuse of language an element  $x\in \mathcal{H}(N)^n$ refers  to $x=(x_1,..,x_{n})$, where $x_i\in \mathcal{H}(N)$,  and  $\alpha(x)$ denotes $(\alpha (x_1),...,\alpha (x_n))$.
\end{definition}

\subsection{From $n$-ary Hom-Nambu-Lie superalgebra to Hom-Leibniz superalgebra}

Let $(N,[\cdot ,...,\cdot ],\alpha)$ be an $n$-ary multiplicative Hom-Nambu-Lie  superalgebra.  On  $\wedge^{n-1}N$ which is the set of elements $x_1\wedge...\wedge x_{ n-1}$ that are skew-symmetric in their arguments,  we define, for $x=x_1\otimes...\otimes x_{ n-1}\in\otimes^{n-1}N$,
 a linear map $\hat{\alpha}
:\otimes^{n-1}N\longrightarrow\otimes^{n-1}N$ for all $x=x_1\otimes...\otimes x_{ n-1}\in\otimes^{n-1}N$, by \begin{equation}\label{mapLeibniz}\hat{\alpha}
(x)=\alpha(x_1)\otimes...\otimes\alpha(x_{n-1})\,\end{equation}
and an even  bilinear map $[\ ,\ ]_{\alpha}:\wedge^{n-1}N\times\wedge^{n-1}N\longrightarrow\wedge^{n-1}N$  defined for all $ x,y\in \otimes^{n-1}\mathcal{H}(N)$ by
\begin{equation}\label{brackLei}[x ,y]_{\alpha}=L(x)\bullet_{\alpha}y=\sum_{i=0}^{n-1}(-1)^{|x||y^i|}\big(\alpha(y_1),...,L(x)\cdot y_i,...,\alpha(y_{n-1})\big).\end{equation}

We denote by $\mathcal{L}(N)$ the space $\wedge^{n-1}N$ and  call it  the fundamental set, $\mathcal{L}(N)$ is $\mathbb{Z}_2$-graded.
\begin{lem}\label{3.1}
The map $L$ satisfies
\begin{equation} L([x ,y ]_{\alpha})\cdot \alpha(z)=L(\alpha(x))\cdot \big(L(y)\cdot z\big)-(-1)^{|x||y|}L(\alpha(y))\cdot \big(L(x)\cdot z\big)\end{equation}
for all $x,\ y\in \mathcal{H}(\mathcal{L}(N)),\ z\in \mathcal{H}(N).$
\end{lem}
\begin{proposition}\label{HomLeibOfHomNambu}The triple $\big(\mathcal{L}(N),\ [\ ,\ ]_{\alpha},\ \widetilde{\alpha}\big)$ is a Hom-Leibniz superalgebra.
\end{proposition}
\begin{proof} Let $x=x_1\wedge...\wedge x_{ n-1},\ y=y_1\wedge...\wedge y_{n-1} $ and $u=u_1\wedge...\wedge u_{n-1}\in \mathcal{H}(\mathcal{L}(N))$, the Super-Leibniz identity \eqref{super-LeiIdent} can be written
\begin{equation}\label{super-LeiIdent2}\big[[x ,y ]_{\alpha} ,\widetilde{\alpha}(u)\big]_{\alpha}=[\widetilde{\alpha}(x) ,[y ,u ]_{\alpha} \big]_{\alpha}-(-1)^{|x||y|}[\widetilde{\alpha}(y) ,[x ,u ]_{\alpha}\big]_{\alpha}\end{equation}
and  equivalently for $v\in \mathcal{H}(N)$
\begin{equation}\label{brackLei3}
\Big(L\big(L(x)\bullet_{\alpha}y\big)\bullet_{\alpha}\tilde{\alpha}(u)\Big)\cdot (v)
=\Big(L(\alpha(x))\bullet_{\alpha}\big(L(y)\bullet_{\alpha}u\big)\Big)\cdot (v)-(-1)^{|x||y|}
\Big(L(\alpha(y))\bullet_{\alpha}\big(L(x)\bullet_{\alpha}u\big)\Big)\cdot (v).
\end{equation}
Let us compute first $\Big(L(\tilde{\alpha}(x))\bullet_{\alpha}\big(L(y)\bullet_{\alpha}u\big)\Big)$. This is given by
\begin{align*}
\Big(L(\alpha(x))\bullet_{\alpha}\big(L(y)\bullet_{\alpha}u\big)\Big) &= \sum_{i=1}^{n-1}(-1)^{|y||u^i|}L(\alpha(x))\bullet_{\alpha}\big(\alpha(u_1),...,L(y)\cdot u_i,...,\alpha(u_{n-1})\big) \\
&= \sum_{i=1}^{n-1}\sum_{j<i,j=1}^{n-1}(-1)^{|x||u^j|+|y||u^i|}\big(\alpha^2(u_1),...,\alpha(L(x)\cdot u_j),...,\alpha(L(y)\cdot u_i)...,\alpha^2(u_{n-1})\big) \\
&+\sum_{i=}^{n-1}\sum_{j>i,j=1}^{n-1}(-1)^{|x||y|+|x||u^j|+|y||u^i|}\big(\alpha^2(u_1),...,\alpha(L(y)\cdot u_i)...,\alpha(L(x)\cdot u_j),...,\alpha^2(u_{n-1})\big) \\
&+ \sum_{i=1}^{n-1}(-1)^{|x||u^i|+|y||u^i|}\big(\alpha^2(u_1),...,L(\tilde{\alpha}(x))\cdot (L(y)\cdot u_i),...,\alpha^2(u_{n-1})\big).
\end{align*}
                                The right hand side of \eqref{brackLei3} is skew-symmetric in $x$, $y$. Hence,
                                $$
    \Big(L(\alpha(x))\bullet_{\alpha}\big(L(y)\bullet_{\alpha}u\big)\Big)-(-1)^{|x||y|}
\Big(L(\alpha(y))\bullet_{\alpha}\big(L(x)\bullet_{\alpha}u\big)\Big)=$$\begin{equation}
\sum_{i=1}^{n-1} (-1)^{|x||u^i|+|y||u^i|} (\alpha^2(u_1),...,\{L(\alpha(x))\cdot (L(y)\cdot u_i)
  - (-1)^{|x||y|}L(\alpha(y))\cdot (L(x)\cdot u_i)\},...,\alpha^2(u_{n-1})\big).
\end{equation}
In the other hand, using Definition \eqref{brackLei}, we find
\begin{align*}
&\Big(L\big(L(x)\bullet_{\alpha}y\big)\bullet_{\alpha}\tilde{\alpha}(u)\Big)
\\&= \sum_{i=1}^{n-1}\sum_{j=1}^{n-1}(-1)^{|x||y^j|+|x||u^i|+|y||u^i|}\big(\alpha^2(u_1),...,\alpha^2(u_{i-1}),\\ & \quad \quad \quad [\alpha(y_1),...,
L(x)\cdot y_j,...,\alpha(y_{n-1}),\alpha(u_i)],\alpha^2(u_{i+1}),...,\alpha^2(u_{n-1})\big)\\
&=\sum_{i=0}^{n-1}(-1)^{|x||u^i|+|y||u^i|}\big(\alpha^2(u_1),...,\alpha^2(u_{i-1}),[x ,y ]_{\alpha}\cdot \alpha(u_i),\alpha^2(u_{i+1}),...,\alpha^2(u_{n-1})\big).\end{align*}
The identity \eqref{super-LeiIdent2} holds by using Lemma \ref{3.1}.
\end{proof}
\begin{remark}
We obtain a similar result if we consider the space $TN=\otimes^{n-1} N$ instead of $\mathcal{L}(N)$.
\end{remark}

\subsection{Representations of Hom-Nambu-Lie superalgebras}
We provide in the following a graded version of the study of representations of $n$-ary Hom-Nambu-Lie
algebras stated in \cite{AmmarSamiMakhloufNov2010}.
\begin{definition}
A representation of an $n$-ary Hom-Nambu  superalgebra $(N,[\cdot ,...,\cdot ],\tilde{\alpha})$
on a $\mathbb{Z}_2$-graded vector space $V=V_0\oplus V_1$ is an even super skew-symmetric multilinear map $\rho:N^{ n-1}\longrightarrow End(V)$
satisfying  for $x,y\in \mathcal{H}(N^{n-1})$ the identity
\begin{equation}\label{RepIdentity1}
\rho(\tilde{\alpha}(x))\circ\rho(y)-(-1)^{|x||y|}\rho(\tilde{\alpha}(y))\circ\rho(x)=\sum_{i=1}^{n-1}(-1)^{|y||x^i|}\rho(\alpha_1 (x_1),...,L(y)\cdot x_i,...,\alpha_{n-2} (x_{n-1}))\circ \nu,
\end{equation}
where $\nu$ is  an even endomorphism on $V$.
We denote this representation by a triple $(V,\rho,\nu)$.
\end{definition}

Two representations $(V,\rho,\nu )$ and $(V',\rho',\nu' )$ of $N$ are \emph{equivalent} if there exists $f:V \rightarrow V' $, an isomorphism of vector space, such that $f(x\cdot v)=x\cdot ' f(v)$ and $f\circ \nu =\nu' \circ f$ where $x\cdot v=\rho(x)(v)$ and $x\cdot' v'=\rho'(x)(v')$ for $x\in \mathcal{H} (N^{n-1})$, $v\in V$ and $v'\in V'$. Then $V $ and $V'$ are viewed as  $N^{n-1}$-modules.
\begin{remark}
Let $(N,[\cdot ,...,\cdot ],\alpha)$ be a multiplicative  $n$-ary Hom-Nambu  superalgebra. The identity \eqref{RepIdentity1} can be written
\begin{equation}\label{RepIdentity2}
\rho(\tilde{\alpha}(x))\circ\rho(y)-(-1)^{|x||y|}\rho(\tilde{\alpha}(y))\circ\rho(x)=\rho([x,y]_\alpha)\circ \nu.
\end{equation}
\end{remark}
\begin{example}Let $(N, [\cdot  ,..., \cdot ],  \tilde{\alpha} )$ be an $n$-ary Hom-Nambu-Lie superalgebra. The map $ L$  defined in \eqref{adjointMapNaire} is a representation on $N$,  where the endomorphism  $\nu$ is the twist map $\alpha_{n-1}$. The identity \eqref{RepIdentity1} is equivalent to the Hom-Nambu identity \eqref{NambuIdentity}. This representation is called the \emph{adjoint} representation.
\end{example}

\subsection{Cohomology of Hom-Nambu-Lie superalgebras }

  Let $(N,[\cdot ,...,\cdot ],\alpha)$ be an $n$-ary  multiplicative Hom-Nambu-Lie superalgebra ($\alpha_1=...=\alpha_n=\alpha$) and $(V,\rho,\nu)$  a $N^{n-1}$-module.\\

  \begin{definition}
A $p$-cochain is a $(p+1)$-linear map
 $
  \varphi:\mathcal{L}(N)\otimes...\otimes \mathcal{L}(N)\wedge N \longrightarrow V $,
such that $$\nu\circ\varphi(x_1,...,x_p,z)=\varphi(\alpha(x_1),...,\alpha(x_p),\alpha(z)).$$
We denote the set of  $p$-cochains by $\mathcal{C}^p_{\alpha,\mu}(N,V)$
\end{definition}
  \begin{definition}We call, for $p\geq 1$, $p$-coboundary operator of the multiplicative $n$-ary Hom-Nambu-Lie superalgebra $(N,[\cdot ,...,\cdot ],\alpha)$ the linear map $\delta^p:\mathcal{C}^p_{\alpha,\mu}(N,V)\rightarrow \mathcal{C}^{p+1}_{\alpha,\mu}(N,V)$ defined by

  \begin{align}
                                    &\label{Nambucohomo}\delta^p\psi(x_1,...,x_p,x_{p+1},z) \\&=\sum_{1\leq i< j}^{p+1}(-1)^{i+(|x_{i+1}|+...+|x_{j-1}|)|x_i|}\psi\big(\alpha(x_1),...,\widehat{\alpha(x_i)},...,
                                  \alpha(x_{j-1}),[x_i,x_j]_{\alpha},...,\alpha(x_{p+1}),\alpha(z)\big)\nonumber\\
                                    &+ \sum_{i=1}^{p+1}(-1)^{i+(|x_{i+1}|+...+|x_{p+1}|)|x_i|}\psi\big(\alpha(x_1),...,\widehat{\alpha(x_i)},...,
                                    \alpha(x_{p+1}),L(x_i)\cdot z\big) \nonumber\\
                                    &+ \sum_{i=1}^{p+1}(-1)^{i+1+(|\psi|+|x_{1}|+...+|x_{i-1}|)|x_i|}\rho(\alpha^p(x_i))\Big( \psi\big(x_1,...,\widehat{x_i},...,
                                  x_{p+1},z\big)\Big)\nonumber\\
                                  &+(-1)^{(|\psi|+|x_{1}|+...+|x_{p+1}|)|z|} \big(\psi(x_1,...,x_p,\ )\cdot x_{p+1}\big)\bullet_{\alpha}\alpha^p(z)\nonumber
                                \end{align}
where
  \begin{align}&\big(\psi(x_1,...,x_p,\ )\cdot x_{p+1}\big)\bullet_{\alpha}\alpha^p(z)\nonumber\\
  &=\dl\sum_{i=1}^{p-1}(-1)^{i+|\psi|+|x_1|+...+|x_p|+|x_{p+1}^i|+|z| }\rho(\alpha^p(x_{p+1}^1),...,\alpha^p(x_{p+1}^{n-1}),\alpha^p(z))(\psi(x_1,...,x_p,x_{p+1}^i )),\end{align}

  for $x_i=(x_i^j)_{1\leq j\leq n-1}\in \mathcal{H}(\mathcal{L}(N)),\ 1\leq i\leq p+1$, $z\in \mathcal{H}(N)$.
  \end{definition}

  \begin{proposition}\label{opercob}Let $\psi\in \mathcal{C}^p_{\alpha,\mu}(N,V)$ be a $p$-cochain then
  $$\delta^{p+1}\circ\delta^p(\psi)=0.$$
  \end{proposition}
Therefore, we have a cohomology complex. Let $\psi$ be a $p$-cochain, $x_i=(x_i^j)_{1\leq j\leq n-1}\in \mathcal{H}(\mathcal{L}(N)),\ 1\leq i\leq p+2$ and $z\in \mathcal{H}(N)$.  
  \begin{definition}We define the space of
  \begin{itemize}
    \item
   $p$-cocycles  by
  $$\mathcal{Z}^p_{\alpha,\nu}(N,V)=\{\varphi\in \mathcal{C}^p_{\alpha,\nu}(N,V):\delta^p\varphi=0\},$$
   \item  $p$-coboundaries  by
  $$\mathcal{B}^p_{\alpha,\nu}(N,V)=\{\psi=\delta^{p-1}\varphi:\varphi\in \mathcal{C}^{p-1}_{\alpha,\nu}(N,V).$$\end{itemize}
  \end{definition}
  \begin{lem}$\mathcal{B}^p_{\alpha,\nu}(N,V)\subset \mathcal{Z}^p_{\alpha,\nu}(N,V)$.
  \end{lem}
  \begin{definition}We call the $p^{\textrm{th}}$-cohomology group  the quotient
  $$\mathcal{H}^p_{\alpha,\nu}(N,V)=\mathcal{Z}^p_{\alpha,\nu}(N,V)/\mathcal{B}^p_{\alpha,\nu}(N,V).$$
  \end{definition}

\subsection{Cohomology of $n$-ary Hom-superalgebras induced by cohomology of Hom-Leibniz superalgebras}
In this section we extend to $n$-ary multiplicative Hom-Nambu-Lie superalgebras the Takhtajan's construction of a cohomology of ternary Nambu-Lie algebras starting from Chevalley-Eilenberg cohomology of  Lie algebras, (see \cite{Takhtajan0,Takhtajan1,Takhtajan2}). \\

Let $(N,[\cdot ,...,\cdot ],\alpha)$ be a multiplicative $n$-ary Hom-Nambu-Lie superalgebra and the triple $(\mathcal{L}(N)=N^{\otimes n-1},[\cdot,\cdot]_\alpha,\alpha)$  be the Hom-Leibniz superalgebra associated to $N$ where the bracket is defined in \eqref{brackLei}.

\begin{thm}\label{constakh}Let $(N,[\cdot ,...,\cdot ],\alpha)$ be a multiplicative $n$-ary Hom-Nambu-Lie superalgebra and $\mathcal{C}^p_{\alpha,\alpha}(N,N)=Hom(\otimes^{p}\mathcal{L}(N)\otimes N,N)$ for $p\geq 1$ be the set of cochains. Let
$\Delta :\mathcal{C}^p_{\alpha,\alpha}(N,N)\rightarrow C^p_{\widetilde{\alpha},\widetilde{\alpha}}(\mathcal{L}(N),\mathcal{L}(N))$ be the linear map defined for $p=0$ and $(x_1,...,x_{n-1})\in \mathcal{H}(N)^{n-1}$ by \begin{eqnarray}\Delta\varphi(x_1\otimes\cdots\otimes x_{n-1})=\dl\sum_{i=0}^{n-1}(-1)^{(|x_{1}|+...+|x_{i-1}|)|\varphi|}x_1\otimes\cdots\otimes\varphi(x_i)\otimes\cdots\otimes x_{n-1}\end{eqnarray} and for $p>0$ by
\begin{align}
&\label{defdelta}(\Delta\varphi )(a_1,\cdots ,a_{p+1})=\\ & \nonumber  \sum_{i=1}^{n-1}(-1)^{(|\varphi|+|a_{1}|+...+|a_{p}|)|a_{p+1}^{i-1}|}\alpha^{p-1}(x_{p+1}^1)\otimes\cdots\otimes\varphi(a_1,\cdots ,a_{n-1}\otimes x_{p+1}^i)\otimes\cdots\otimes \alpha^{n-1}(x_{p+1}^{n-1}),\nonumber
\end{align}
where we set $a_j=x_{j}^1\otimes\cdots\otimes x_j^{n-1}.$

Then there exists a cohomology complex $(\mathcal{C}_{\alpha,\alpha}^\bullet(N,N),\delta )$ for $n$-ary Hom-Nambu-Lie superalgebras such that $$d\circ \Delta =\Delta\circ \delta.$$

The coboundary map $\delta: \mathcal{C}^p_N(N,N)\rightarrow \mathcal{C}^{p+1}_N(N,N)$ is defined for $\varphi\in \mathcal{C}^p_N(N,N)$ by
\begin{align*}
                                    &\delta^p\psi(a_1,...,a_p,a_{p+1},z) \\ &=\sum_{1\leq i< j}^{p+1}(-1)^i(-1)^{(|a_{i+1}|+...+|a_{j-1}|)|a_i|}\psi\big(\alpha(a_1),...,\widehat{\alpha(a_i)},...,
                                  \alpha(a_{j-1}),[a_i,a_j]_{\alpha},...,\alpha(a_{p+1}),\alpha(z)\big)\nonumber\\
                                    &+ \sum_{i=1}^{p+1}(-1)^i(-1)^{(|a_{i+1}|+...+|a_{p+1}|)|a_i|}\psi\big(\alpha(a_1),...,\widehat{\alpha(a_i)},...,
                                    \alpha(a_{p+1}),L(a_i)\cdot z\big) \nonumber\\
                                    &+ \sum_{i=1}^{p+1}(-1)^{i+1}(-1)^{(|\psi|+|a_{1}|+...+|a_{i-1}|)|a_i|}L(\alpha^p(a_i))\cdot \psi\big(a_1,...,\widehat{a_i},...,
                                  a_{p+1},z\big)\nonumber\\
                                  &+(-1)^{(|\psi|+|a_{1}|+...+|a_{p+1}|)|z|} (-1)^p\big(\psi(a_1,...,a_p,\ )\cdot a_{p+1}\big)\bullet_{\alpha}\alpha^p(z)\nonumber
                                \end{align*}
where
  $$\big(\varphi(a_1,...,a_p,\ )\cdot a_{p+1}\big)\bullet_{\alpha}\alpha^p(z)=\dl\sum_{i=1}^{n-1}[\alpha^p(x_{p+1}^1),...,\varphi(a_1,...,a_p,x_{p+1}^i ),...,\alpha^p(x_{p+1}^{n-1}),\alpha^p(z)].$$\\for  $a_i\in \mathcal{H}(\mathcal{L}(N))$, $x\in \mathcal{H}( N)$.
\end{thm}

\begin{proof}Let $\varphi\in\mathcal{C}^p_{\alpha,\alpha}(N,N)$ and $(a_1\cdots a_{p+1})\in\mathcal{L} $ where $a_j=x_{1}^j\otimes\cdots\otimes x_{n-1}^j.$\\
 Then $\Delta\varphi\in C^p_{\widetilde{\alpha},\widetilde{\alpha}}(\mathcal{L}(N),\mathcal{L}(N))$ and according to \eqref{Leibnizcohomo} and \eqref{Nambucohomo} we set $d=d_1+d_2+d_3$ and $\delta=\delta_1+\delta_2+\delta_3+\delta_4$ , where

\begin{align}
d_1\varphi(a_1,\cdots ,a_{p+1})& =\sum_{k=1}^p(-1)^{k-1} \big[ \alpha^{p-1}(a_k),\varphi(a_1,\cdots,\widehat{a_k},\cdots,a_{p+1})\big]\nonumber\\
d_2\varphi(a_1,\cdots ,a_{p+1})\ & =(-1)^{p+1} \big[ \varphi(a_1\otimes\cdots\otimes a_p),\alpha^{p-1}(a_{p+1}) \big]  \nonumber\\
d_3\varphi(a_1,\cdots ,a_{p+1})\ & =\sum_{1\leq k<j}^{p+1}{(-1)^k \varphi(\alpha(a_1)\otimes \cdots \otimes\widehat{a_k}\otimes \cdots\otimes\alpha(a_{j-1})\otimes [a_k,a_j]\otimes\alpha(a_{j+1})\otimes \cdots\otimes\alpha(a_{p+1}))}\nonumber
\end{align}
and
\begin{align*}
 &  \delta_1^p\psi(x_1,...,x_{p+1},z )= \sum_{1\leq i< j}^{p+1}(-1)^{i+(|x_{i+1}|+...+|x_{j-1}|)|x_i|}\psi\big(\alpha(x_1),...,\widehat{\alpha(x_i)},...,
                                  \alpha(x_{j-1}),[x_i,x_j]_{\alpha},...,\alpha(x_{p+1}),\alpha(z)\big) \\
&  \delta_2^p\psi(x_1,...,x_{p+1},z) =  \sum_{i=1}^{p+1}(-1)^{i+(|x_{i+1}|+...+|x_{p+1}|)|x_i|}\psi\big(\alpha(x_1),...,\widehat{\alpha(x_i)},...,
                                   \alpha(x_{p+1}),L(x_i)\cdot z\big) \\
 &  \delta_3^p\psi(x_1,...,x_{p+1},z) = \sum_{i=1}^{p+1}(-1)^{i+1+(|\psi|+|x_{1}|+...+|x_{i-1}|)|x_i|}\rho(\alpha^p(x_i))\Big(   \psi\big(x_1,...,\widehat{x_i},...,
                                  x_{p+1},z\big)\Big) \\
&   \delta_4^p\psi(x_1,...,x_{p+1},z) =  (-1)^{(|\psi|+|x_{1}|+...+|x_{p+1}|)|z|} \big(\psi(x_1,...,x_p,\ )\cdot x_{p+1}\big)\bullet_{\alpha}\alpha^p(z).
  \end{align*}
By \eqref{defdelta} we have
\begin{eqnarray*}
& & d_1\circ\Delta\varphi(a_1,\cdots ,a_{p+1})\\
  & =&\sum_{k=1}^p(-1)^{k-1}(-1)^{(|\varphi|+|a_1|+...+|a_{k-1}|)|a_k|} \big[ \alpha^{p-1}(a_k),\Delta\varphi(a_1,\cdots,\widehat{a_k},\cdots,a_{p+1})\big]\\
  & =&\sum_{k=1}^p(-1)^{k-1}(-1)^{(|\varphi|+|a_1|+...+|a_{k-1}|)|a_k|}\sum_{i=1}^{n-1} (-1)^{(|\varphi|+|a_{1}|+...+\widehat{|a_{k}|}+...+|a_{p}|)|a_{p+1}^{i-1}|}\\
&  &\big[ \alpha^{p-1}(a_k),\alpha^{p-1}(x_{p+1}^1)\otimes\cdots\otimes\varphi(a_1,\cdots,\widehat{a_k},\cdots,x_{p+1}^i)
 \otimes\cdots\otimes\alpha^{p-1}(x_{p+1}^{n-1})\big]\\
  & =&\sum_{k=1}^p(-1)^{k-1}(-1)^{(|\varphi|+|a_1|+...+|a_{k-1}|)|a_k|}\sum_{i>j}^{n-1} (-1)^{(|\varphi|+|a_{1}|+...+|a_{k}|+...+|a_{p}|)|a_{p+1}^{i-1}|}\\
   & &\alpha^{p}(x_{p+1}^1)\otimes\cdots\otimes L(\alpha^{p-1}(a_k)).\alpha^{p-1}(x_{p+1}^j)\otimes\cdots\otimes\varphi(a_1,\cdots,\widehat{a_k},\cdots,x_{p+1}^i)
 \otimes\cdots\otimes\alpha^{p}(x_{p+1}^{n-1})\\
  & +&\sum_{k=1}^p(-1)^{k-1}(-1)^{(|\varphi|+|a_{k+1}|+...+|a_{p}|)|a_k|}\sum_{j>i}^{n-1}(-1)^{(|\varphi|+|a_{1}|+...+|a_{k}|+...+|a_{p}|)|a_{p+1}^{i-1}|}\\
   & &  \alpha^{p}(x_{p+1}^1)\otimes\cdots\otimes \varphi(a_1,\cdots,\widehat{a_k},\cdots,x_{p+1}^i)
 \otimes\cdots \otimes L(\alpha^{p-1}(a_k)).\alpha^{p-1}(x_{p+1}^j)\otimes\cdots\otimes\alpha^{p}(x_{p+1}^{n-1})
 \\ & +&\sum_{k=1}^p(-1)^{k-1}(-1)^{(|\varphi|+|a_1|+...+|a_{k-1}|)|a_k|}\sum_{i=1}^{n-1} (-1)^{(|\varphi|+|a_{1}|+...+|a_{k}|+...+|a_{p}|)|a_{p+1}^{i-1}|}\\
   & &  \alpha^{p}(x_{p+1}^1)\otimes\cdots\otimes L(\alpha^{p-1}(a_k)).\varphi(a_1,\cdots,\widehat{a_k},\cdots,x_{p+1}^i)
 \otimes\cdots\otimes\alpha^{p}(x_{p+1}^{n-1})\end{eqnarray*}
 \begin{eqnarray*}
   & =&\sum_{k=1}^p(-1)^{k-1}(-1)^{(|\varphi|+|a_1|+...+|a_{k-1}|)|a_k|}\sum_{i>j}^{n-1} (-1)^{(|\varphi|+|a_{1}|+...+|a_{k}|+...+|a_{p}|)|a_{p+1}^{i-1}|}\\
   & &\alpha^{p}(x_{p+1}^1)\otimes\cdots\otimes L(\alpha^{p-1}(a_k)).\alpha^{p-1}(x_{p+1}^j)\otimes\cdots\otimes\varphi(a_1,\cdots,\widehat{a_k},\cdots,x_{p+1}^i)
 \otimes\cdots\otimes\alpha^{p}(x_{p+1}^{n-1})\\
  & +&\sum_{k=1}^p(-1)^{k-1}(-1)^{(|\varphi|+|a_{k+1}|+...+|a_{p}|)|a_k|}\sum_{j>i}^{n-1}(-1)^{(|\varphi|+|a_{1}|+...+|a_{k}|+...+|a_{p}|)|a_{p+1}^{i-1}|}\\
   & &  \alpha^{p}(x_{p+1}^1)\otimes\cdots\otimes \varphi(a_1,\cdots,\widehat{a_k},\cdots,x_{p+1}^i)
 \otimes\cdots \otimes L(\alpha^{p-1}(a_k)).\alpha^{p-1}(x_{p+1}^j)\otimes\cdots\otimes\alpha^{p}(x_{p+1}^{n-1})
 \\
    & +&\Delta\circ\delta_3\circ\varphi(a_1,\cdots ,a_{p+1})\\
  & =&\Lambda_1+\Lambda_2+\Delta\circ\delta_3\circ\varphi(a_1,\cdots ,a_{p+1})\nonumber
\end{eqnarray*}
where
\begin{align*}
 & \Lambda_1=\sum_{k=1}^p(-1)^{k-1}(-1)^{(|\varphi|+|a_1|+...+|a_{k-1}|)|a_k|}\sum_{i>j}^{n-1} (-1)^{(|\varphi|+|a_{1}|+...+|a_{k}|+...+|a_{p}|)|a_{p+1}^{i-1}|}\\
   & \alpha^{p}(x_{p+1}^1)\otimes\cdots\otimes L(\alpha^{p-1}(a_k)).\alpha^{p-1}(x_{p+1}^j)\otimes\cdots\otimes\varphi(a_1,\cdots,\widehat{a_k},\cdots,x_{p+1}^i)
 \otimes\cdots\otimes\alpha^{p}(x_{p+1}^{n-1})\\
  & \Lambda_2=\sum_{k=1}^p(-1)^{k-1}(-1)^{(|\varphi|+|a_{k+1}|+...+|a_{p}|)|a_k|}\sum_{j>i}^{n-1}(-1)^{(|\varphi|+|a_{1}|+...+|a_{k}|+...+|a_{p}|)|a_{p+1}^{i-1}|}\\
   &   \alpha^{p}(x_{p+1}^1)\otimes\cdots\otimes \varphi(a_1,\cdots,\widehat{a_k},\cdots,x_{p+1}^i)
 \otimes\cdots \otimes L(\alpha^{p-1}(a_k)).\alpha^{p-1}(x_{p+1}^j)\otimes\cdots\otimes\alpha^{p}(x_{p+1}^{n-1}).
\end{align*}
Similarly we can prove that
\begin{align}
 d_2\circ\Delta\varphi(a_1,\cdots ,a_{p+1})
 \ & =\Delta\circ\delta_4\varphi(a_1,\cdots ,a_{p+1})
\end{align}
and
\begin{align}
& d_3\Delta\circ\varphi(a_1,\cdots ,a_{p+1})=\Delta\circ\delta_1\varphi(a_1,\cdots ,a_{p+1})+\Delta\circ\delta_2\varphi(a_1,\cdots ,a_{p+1})-\Lambda_1-\Lambda_2.
\end{align}
Finally we have
$$d\circ\Delta=d_1\circ\Delta+d_2\circ\Delta+d_3\circ\Delta=\Delta\circ\delta_3+\Delta\circ\delta_4+
\Delta\circ\delta_1+\Delta\circ\delta_2=\Delta\circ\delta$$
where $\delta=\delta_1+\delta_2+\delta_3+\delta_4$ as defined in Proof  \ref{opercob}.
\end{proof}
\begin{remark}If $d^2=0$, then $\delta^2=0$.\\
In fact, since $d\circ\Delta=\Delta\circ\delta$, then $$\Delta\circ\delta^2 =\Delta\circ\delta\circ\delta =d\circ\Delta\circ\delta=d\circ. d\circ\Delta=d^2\circ\Delta=0.$$
\end{remark}

\bibliographystyle{amsplain}
\providecommand{\bysame}{\leavevmode\hbox to3em{\hrulefill}\thinspace}
\providecommand{\MR}{\relax\ifhmode\unskip\space\fi MR }
\providecommand{\MRhref}[2]{%
  \href{http://www.ams.org/mathscinet-getitem?mr=#1}{#2}
}
\providecommand{\href}[2]{#2}

\end{document}